%% file: SZ_arXiv2.tex
\documentclass{article}

\input{SZ_short_sicon_shared.tex}

\begin{document}

\maketitle

% REQUIRED
\begin{abstract}
In this paper we study stochastic control problems with delayed information, that is, the control at time $t$ can depend only on the information observed before time $t-\ch$ for some delay parameter $\ch$. Such delay occurs frequently in practice and  can be viewed as a special case of partial observation. When the time duration $T$ is smaller than $\ch$, the problem becomes a deterministic control problem in stochastic setting. While seemingly simple, the problem involves certain time inconsistency issue, and the value function naturally relies on the distribution of  the state process and thus is a solution to a nonlinear master equation. Consequently, the optimal state process solves a McKean-Vlasov SDE. In the general case that $T$ is larger than $\ch$, the master equation becomes path-dependent and the corresponding  McKean-Vlasov SDE involves the conditional distribution of the state process. We shall build these connections rigorously, and obtain the existence of classical solution of these nonlinear (path-dependent) master equations in some special cases.
\end{abstract}

% REQUIRED


\textbf{AMS:} 60H30, 93E20.

\section{Introduction}
Consider a stochastic control problem:
\bea
\label{control}
\left.\ba{c}
\dis V_0 = \sup_{\a\in\cA} \dbE\Big[g(X^\a_T) + \int_0^T f(t, X^\a_t, \a_t) dt\Big],\\
\dis \mbox{where}\q  X^\a_t = x + \int_0^t b(s, X^\a_s, \a_s) ds + \int_0^t \si(s, X^\a_s, \a_s) dW_s,
\ea\right.
\eea
and $\cA$ is an appropriate set of $A$-valued admissible controls. It is well known that, under mild conditions, $V_0 = u(0, x)$, where $u$ is the solution of an HJB equation. One standard but crucial condition in the literature is that the admissible control is $\dbF$-progressively measurable, where $\dbF = \{\cF_t\}_{0\le t\le T}$ is a filtration under which $W$ is a Brownian motion. 

Our paper is mainly motivated by the following practical consideration. Note that $\cF_t$ stands for the information the player observes over time period $[0, t]$. In many practical situations, the player needs some time to collect and/or to analyze the information, including numerical computations. Thus, the control $\a_t$ the player needs to act at time $t$ may not be able to utilize the most recent information, or say, there is some information delay. To be precise, let $\ch>0$ be a fixed constant standing for the delay parameter. In this paper we shall study the control problem \reff{control} by restricting the admissible control $\a$ in $\cA^\ch_0$, i.e. such that  $\a_t\in \cF_{(t-\ch)^+}$, for all $t\in [0, T]$. This can be viewed as a special case of stochastic controls with \textit{partial observation}. For a literature review, see Section \ref{sec:review}.

We first consider the simple case that $T \le \ch$. Then $\a_t\in \dbL^0(\cF_0)$ for all $t\in [0, T]$, and thus this is a deterministic control problem (assuming $\cF_0$ is degenerate), but in a stochastic framework. While seemingly simpler, the constraint that the control is deterministic actually makes the problem more involving. The main reason is that such a problem is time inconsistent if one follows the standard approach. Intuitively, for problem \reff{control} the optimal control $\a^*_t = \a^*(t, X^*_t)$ may typically depend on the corresponding state process $X^*_t$ and thus is random. If the control is deterministic, the optimal control (assuming its existence) $\a^*_t = \a^*(t, x)$ should typically depend only on the initial value $x$ for all $t\in [0, T]$. When one considers a dynamic problem over time period $[t_0, T]$, the new ``deterministic" optimal control will become $\tilde \a^*_t =  \tilde \a^*(t, X^*_{t_0})$ for all $t\in [t_0, T]$, which will be $\cF_{t_0}$-measurable rather than $\cF_0$-measurable, and thus typically $\tilde \a^*_t \neq \a^*_t$ for $t\ge t_0$. That is, the problem is time inconsistent.  

We aim to solve the problem in a time consistent way. Note again that, in standard control problem, the optimal control reacts to the state process $X^*_t$. If the control is deterministic, and we still want the optimal one to react to the state process in some way, the most natural choice would be that $\a^*_t$ reacts to the law of $X^*_t$. This is indeed true. At time $t_0$, instead of specifying a value of $X_{t_0}$, we shall specify the distribution $\mu$ of $X_{t_0}$ and define the value $V(t_0, \mu)$ for optimization over $[t_0, T]$ with deterministic control. It turns out that this dynamic problem is time consistent. The function $V$ satisfies an appropriate dynamic programming principle and is the solution of a so-called master equation. Moreover, $V_0 = V(0, \d_x)$, where $x$ is the initial value $X_0$ and $\d_x$ is the Dirac-measure of $x$.  

To understand the master equation, we remark that $V: [0, T] \times \cP_2(\dbR^d)\to \dbR$ is a deterministic function, where $\cP_2(\dbR^d)$ is the set of square-integrable probability measures on $\dbR^d$. It is known that the derivative of $V$ in terms of $\mu\in \cP_2(\dbR^d)$ takes the form $\pa_\mu V: (t, \mu, x)\in [0, T]\times \cP_2(\dbR^d) \times \dbR^d \to \dbR^d$.  Denote by $\pa_x \pa_\mu V$ the standard derivative of $\pa_\mu V$ with respect to $x$. Then the optimization problem \reff{control} with deterministic control is associated with the following HJB type of master equation:
\bea
\label{master-intro}
 \pa_t V(t, \mu) +  H(t, \mu, \pa_\mu V, \pa_x \pa_\mu V) =0,\q V(T, \mu) = \dbE[g(\xi)],
\eea
where, for $p: [0, T]\times \cP_2(\dbR^d)\times \dbR^d \to \dbR^d$ and $q: [0, T]\times \cP_2(\dbR^d)\times \dbR^d \to \dbR^{d\times d}$,
\bea
\label{Hamiltonian-intro}
\left.\ba{c}
\dis H(t, \mu, p, q):= \sup_{a\in A} h(t, \mu, p, q, a),\\  
 h(t, \mu, p, q, a) := \dbE\Big[ b(t, \xi, a) \cd p(t, \mu, \xi) + \dis {1\over 2} \si\si^\top(t, \xi, a) :  q(t,\mu, \xi) + f(t, \xi, a)\Big].
 \ea\right.
\eea
Here $\xi$ is a random variable with law $\mu$. We shall prove the existence of classical solutions for a special case of \reff{master-intro}, which to our best knowledge is new in the literature.  

Assume further that the Hamiltonian $H$ has optimal argument $a^* = I(t, \mu)$ for some function $I: [0, T] \times \cP_2(\dbR^d) \to A$, then the optimal control is $\a^*_t = I(t, \cL_{X^*_t})$, where $\cL_\xi$ is the law of the random variable $\xi$ and $X^*$ solves  the following McKean-Vlasov SDE (assuming its wellposedness):   
\bea
\label{MV-intro}
X^*_t = x + \int_0^t b\big(s, X^*_s, I(s, \cL_{X^*_s})\big) ds +  \int_0^t \si\big(s, X^*_s, I(s, \cL_{X^*_s})\big) dW_s.
\eea
We shall carry out the verification theorem rigorously when $I$ is continuous.

We finally consider the general case $T > \ch$. In this case, for $t > \ch$, the control $\a_t$ is required to be $\cF_{t-\ch}$-measurable. Motivated by both theoretical and practical considerations, we shall use closed-loop controls, namely $\a_t = \a_t(X_{[0, t-\ch]})$ is $\cF_{t-\ch}$-measurable. Then the value function, $V(t,  \mu_{[0,t]} )$, will be path-dependent in the sense that $\mu_{[0,t]} $ denotes the law of the stopped process $X_{[0,t]}$, and the master equation \reff{master-intro} becomes a path-dependent equation. Consequently, the McKean-Vlasov SDE \reff{MV-intro} will involve the conditional law of $X^*_t$.

We finish this section with a thoroughly comparison of different problems and methods that relates to the ones proposed here, which is further developed in Appendix \ref{sect-appendix_comparison}. The rest of the paper will be organized as follows. We discuss the deterministic case in Section \ref{sect-deterministic}. Moreover, a special case is fully developed in Section \ref{sect-example} and the general theory is presented in Section \ref{sect-general}.

\subsection{Comparison to Similar Control Problems and Methods}\label{sec:review}

As mentioned above, problem (\ref{control}) with $\a \in \cA^\ch_0$ might be seen as a special case of stochastic controls with partial observation. Generally, these stochastic control problems assume that the  admissible controls are adapted to a smaller filtration $\dbG$, i.e. $\cG_t \subset \cF_t$, for all $t \in [0,T]$. Few papers have tackled this problem under this generality, see, for instance, \cite{partial_obs_cristopeit1980}, where the existence of the optimal control is studied, and \cite{partial_obs_oksendal2007}, where a maximum principle was derived under the more general L\'evy processes.

Additionally to the delayed case studied in our paper, a very important example of the partial observation problem is the case of noisy observation. This situation has drawn significantly more attention than the other types of partially observed systems. For references, see \cite{partial_obs_bensoussan_book, partial_obs_fleming_pardoux, partial_obs_fleming, partial_obs_fleming2, partial_obs_bismut, partial_obs_tang}, and the more recent \cite{partial_obs, partial_obs_2}. On these aforesaid references, a \textit{separated} control problem is proposed and studied. The optimal control problem of the partially observed system is connected to this separated problem, which is completely observed, using stochastic nonlinear filtering. It is worth noticing that, similarly to what we have found in the control with delayed information, the state variable of the separated control problem is an unnormalized conditional distribution measure and the class of admissible controls is a set of probability measures. Moreover, the dynamics of the aforementioned unnormalized conditional distribution measure is given by the so-called Zakai's equation. 

Moreover, in \cite{partial_obs_2}, the authors have derived, in this context of noisy observation, the dynamic programming principle with flow of probability measures as state variable and the verification theorem of their master equation. Since the deterministic control problem studied in Section \ref{sect-deterministic} is a particular case of the noisy observation problem, our master equation and the dynamic programming principle  in this section could be seen as a special case of theirs. However, our arguments here are much simpler, due to our special setting, and will be important for the general case in Section \ref{sect-general}, so we  decide to report our proofs in details so that the readers can easily grasp the main ideas. 
%Furthermore, additionally to these results, in Section \ref{sect-example}, we study the existence of classical solutions of the nonlinear master equation related to the deterministic control case, which has not been done before.

In a different direction, although analyzing the same control problem as in the references in the paragraph above, \cite{partial_obs_mortensen} and \cite{partial_obs_benes_karatzas} have studied the value of the control problem as a function of the initial conditional probability density and an HJB equation analogous to our master equation (\ref{master-intro}) was derived. Moreover, an It\^o formula for functions of density-valued processes was proved, c.f. Lemma \ref{lem-Ito}. Under the assumption that the agent observes pure independent noise, it turns out that their control problem is equivalent to our deterministic control problem in Section \ref{sect-deterministic}. Moreover, when restricting to only those measures with density,  our master equation \reff{master-intro} is equivalent to Mortensen's HJB equation. In order to verify this, one needs to understand the relation between G\^ateaux derivatives with respect to the probability densities and $\pa_\mu V$, see \cite{master_eq_bensoussan}. For more details, see Appendix \ref{sect-appendix_comparison}.

Furthermore, in the direction of applications of the delayed information setting to Mathematical Finance, \cite{mostafa_delay} have proposed a discrete-time binomial model with delayed information for the price of asset. They studied the super-replication of derivatives with convex payoffs and also the convergence of their model to a continuous-time one (without delay).

A different aspect of delay in control problems is when the control chosen in a previous time, for instance at $t - \ch$, influences the dynamics and/or the cost function at time $t$. In the literature, this is usually called stochastic controls problems with \textit{delay in the control}, see for example, \cite{gozzi_delay_1, gozzi_delay_2, Alekal1971, cheng_delay}. More generally, path dependence in the control was studied in \cite{fito_saporito_control} in the framework of functional It\^o calculus. This type of delay in the control is fundamentally different than the one we study here. Notice that, although the control acts with delay, the agent has full information at time $t$ to choose $\a_t$. This departs completely from the setting we are proposing in this paper. Moreover, as one could easily notice from the aforesaid references, the value function $V$ are not seen as function of probability measures, but a function of the history of the state process. This type of delay in the control was recently applied to the study of systemic risk of a system of banks in \cite{fouque_systemic_delay}.

We remark that the McKean-Vlasov SDE (for forward state process) and the master equation (for backward value function) have received very strong attention in recent years, mainly due to its application in mean field games and systemic risk, see \cite{CHM}, \cite{LL}, as well as \cite{Cardaliaguet},   \cite{BFY},  \cite{CD1, CD2}, and the references therein. In those applications, a large number of players are involved and the measure $\mu$ is introduced to characterize the aggregate behavior of the players. Our motivation here is quite different. We also remark that our paper deals with control problems and the master equation is nonlinear in $\pa_\mu V$ (and/or $\pa_x\pa_\mu V$). For mean field game problems, the master equation involves $V(t,x,\mu)$ and  has quite different nature. On one hand those master equations are nonlocal,  and on the other hand they are typically nonlinear in $\pa_x V$ but linear in $\pa_\mu V$. In fact, in some literature master equations refer to only those for mean field games while the equations for control problems are called HJB equations in Wasserstein space. We nevertheless call both master equations since they share many features.
%We also remark that in this literature typically $V = V(t, X_t, \mu)$ depends on the state process $X$ as well.  For those games, the related master equation could be nonlinear in terms of the (standard) derivative of $V$ in $X_t$, but is always linear in terms of $\pa_\mu V$. Our master equation \reff{master-intro} induced by the control problem \reff{control} is nonlinear in $\pa_\mu V$. 
In a special case, we will prove the existence of classical solutions for the nonlinear master equation \reff{master-intro}.  In general it is difficult to obtain classical solutions for master equations,  some positive results include  \cite{mean_field_associated_pde}, \cite{CDLL},  and \cite{CCD} where the equations are linear in $\pa_\mu V$ and $\pa_x\pa_\mu V$, and \cite{GS} and \cite{BY} where the equations are of first order (without involving $\pa_x \pa_\mu V$). We also refer to \cite{PW} and \cite{WZ} for viscosity solutions of master equations.
%We remark that such existence results are available in first order case (without involving $\pa_x \pa_\mu V$), see e.g.\cite{Cardaliaguet} and \cite{GS}. 
% which to our best knowledge is the first existence result under such nonlinearity. 

Furthermore, although we are considering \textit{control} problems, the delayed observation aspect of our setting is present in  \cite{stackelberg_bensoussan} and \cite{stackelberg_bensoussan_linear_quad} which study Stackelberg stochastic \textit{games} with delayed information.  A simple version of these games can described by two players: a leader and a follower. The leader has full information of both players and the follower has delayed information of the leader state variable (and full information of him/herself). In aforesaid references, the authors study the convergence of the system of $N$-players to its mean field counterpart. Moreover, in the linear-quadratic case, they were able to analyze and derive exact formulas for the mean field game.

\section{The Deterministic Control Problem}
\label{sect-deterministic}
\setcounter{equation}{0}

We remark that this case is the intersection of several related works. For example,  \cite{HT} studied the linear quadratic case by using the stochastic maximum principle, see Appendix  A;  \cite{partial_obs_benes_karatzas} derived a similar equation when the measures have a density, see Appendix B; in particular, our master equation \reff{master}-\reff{Hamiltonian} and the DPP Theorem \ref{thm-DPP1} below are already covered by \cite{partial_obs_2} as a special case. However, since the arguments here are much simpler due to the special structure, which could be helpful for readers to grasp the main ideas, and more importantly since these arguments will be important for the general case in Section \ref{sect-general}, we still provide the details.

Let $T>0$ be a fixed time horizon, $(\Om, \dbF, \dbP)$ a filtered probability space on $[0, T]$, and $W$ an $\dbF$-Brownian motion under $\dbP$. In this section we assume $T \le \ch$ and thus the controls are deterministic. Denote by $\cP_2(\dbR^d)$ the set of square-integrable measures on $\dbR^d$, and for each $\mu\in \cP_2(\dbR^d)$, denote $\dbL^2_\mu(\cF_t) := \{\xi \in \dbL^2(\cF_t): \cL_\xi = \mu\}$, where $\cL_\xi$ denotes the law of $\xi$ and $\dbL^2(\cF_t)$ is the space of  $\cF_t$-measurable square-integrable random variables. %that are measurable with respect to $\cF_t$. 
For technical convenience, we shall assume $\cF_0$ is rich enough such that $\dbL^2_\mu(\cF_0) \neq \emptyset$ for all $\mu \in \cP_2(\dbR^d)$. However, in this and the next section we nevertheless assume the controls are deterministic, rather than $\cF_0$-measurable. Finally, denote $\Th := [0, T] \times \cP_2(\dbR^d)$ and $\ol \Th := \{(t,\xi): t\in [0, T], \xi \in \dbL^2(\cF_t)\}$.

\subsection{The control problem}
 Let $A$ be an (arbitrary) measurable set  in certain Euclidian space, and $\cA_t$ the set of all Borel measurable functions $\a: [t, T] \to A$. For any $(t,\xi) \in \ol\Th$ and $\a\in \cA_t$, define
 \bea
 \label{Ja}
 \left.\ba{c}
\dis X_s^{t, \xi, \a} = \xi + \int_t^s b(r, X_r^{t, \xi, \a}, \a_r)dr + \int_t^s \sigma(r, X_r^{t, \xi, \a}, \a_r) dW_r, ~ s\in [t, T],\\
\dis J(t,\xi, \a) := \dbE\left[g(X_T^{t, \xi, \a}) + \int_t^T f(s, X_s^{t, \xi, \a}, \a_s)ds \right],
\ea\right.
\eea
where $b, \si, f, g$ are deterministic functions with appropriate dimensions. %We shall assume:

\begin{assumption}
\label{assum-deterministic1} (i) $b, \si, f, g$ are measurable in all their variables, and $b(t, 0, a)$, $\si(t,0, a)$, $f(t,0,a)$ are bounded;

(ii) $b$, $\si$ are uniformly Lipschitz continuous in $x$, and uniformly continuous in $t$;

(iii)  $f$ is uniformly continuous in $(t,x)$, and $g$ is uniformly continuous in $x$. 
\end{assumption}

%We remark that the current forms of  $(i)$ and $(iii)$ are for technical convenience and can be weakened.  
Under Assumption \ref{assum-deterministic1}, clearly the SDE in \reff{Ja} is wellposed and 
\bea
\label{Jbound}
|J(t,\xi, \a)| \le C[1+\|\xi\|_{\dbL^2}].
\eea
Moreover, the following result is obvious:

\begin{lemma}
\label{lem-invariant}
Under Assumption \ref{assum-deterministic1},  the mapping $\xi \mapsto J(t,\xi,\a)$ is law invariant.  That is, if $\cL_\xi = \cL_{\xi'}$, then $J(t, \xi, \a) = J(t, \xi', \a)$.
\end{lemma}

We are now ready to introduce the optimization problem:
\bea
\label{Vt}
V(t, \mu) := \sup_{\a\in \cA_t} J(t,\xi, \a),\q (t,\mu) \in \Th~\mbox{and}~ \xi \in \dbL^2_\mu(\cF_t).
\eea
By \reff{Jbound}, $V(t,\mu)$ is finite. We emphasize that $V$ does not depend on the choice of $\xi$, thanks to Lemma \ref{lem-invariant}. 
%By abusing the notation, we also denote
%\bea
%\label{Vtxi}
%V(t,\xi) := V(t, \cL_\xi),\q (t,\xi) \in \ol\Th.
%\eea
Throughout this section, when there is no confusion, for a given $(t,\mu) \in \Th$ we shall always use $\xi$ to denote some random variable in $\dbL^2_\mu(\cF_t)$, and the claimed results will not depend on the choice of $\xi$.

We next establish the dynamic programming principle for $V$. 

%\begin{assum}
%\label{assum-cA1}
%(i) $\cA_t$ contains all $A$-valued constant functions $\a$; % and $V(t, \mu)$ is finite for all $(t,\mu)\in [0, T]\times \cP_2(\dbR^d)$;

%(ii) For any $0\le s < t\le T$ and $\a\in \cA_s$, $\a|_{[t, T]} \in \cA_t$;

%(iii) For any $0\le s < t\le T$ and $\a\in \cA_s, \a'\in \cA_t$, $\a\otimes_t \a' := \a\1_{[s, t)} + \a' \1_{[t, T]}\in \cA_s$. 
%\end{assum}

\begin{theorem}
\label{thm-DPP1}
Let Assumption \ref{assum-deterministic1} hold. Then, for any $(t_1, \mu)\in \Th$, $t_2\in (t_1, T]$, 
\bea
\label{DPP1}
V(t_1, \mu) = \sup_{\a \in \cA_{t_1}} \Big[ V\big(t_2, \cL_{X_{t_2}^{t_1, \xi, \a}}\big) + \int_{t_1}^{t_2} \dbE[f(s, X_s^{t_1, \xi, \a}, \a_s)] ds \Big].
\eea
\end{theorem}

\begin{proof}
For notational simplicity, we assume $t_1 = 0$ and $t_2 = t$, then \reff{DPP1} becomes:
\bea
\label{DPP1-1}
V(0, \mu) =  \widetilde V(0,\mu) := \sup_{\a \in \cA_0} \Big[ V\big(t, \cL_{X_{t}^{0, \xi, \a}}\big) + \int_0^{t} \dbE[f(s, X_s^{0, \xi, \a}, \a_s)] ds \Big].
\eea
On one hand, for any $\a\in \cA_0$, by the flow property for the SDE we have
\beaa
X_s^{0, \xi, \a} = X_s^{t, X_{t}^{0, \xi, \a}, \a'},  \q s\in [t, T],
\eeaa
where $\a' := \a\big|_{[t, T]} \in \cA_{t}$. Then,
\bea
\label{JDPP}
J(0, \xi, \a) &=& \dbE\Big[g(X_T^{t, X_{t}^{0, \xi, \a}, \a'}) + \int_{t}^T f(s, X_s^{t, X_{t}^{0, \xi, \a}, \a'},  \a'_s) ds + \int_0^{t} f(s, X_s^{0, \xi, \a},  \a_s) ds\Big]\nonumber\\
&=& J(t,  X_{t}^{0, \xi, \a}, \a') +  \int_0^{t} \dbE[f(s, X_s^{0, \xi, \a}, \a_s)] ds \\
&\le&  V\big(t,  \cL_{X_{t}^{0, \xi, \a}}\big) +  \int_0^{t} \dbE[f(s, X_s^{0, \xi, \a}, \a_s)] ds \le \widetilde V(0,\mu).\nonumber
\eea
By the arbitrariness of $\a$, we obtain $V(0,\mu) \le \widetilde V(0,\mu)$.

On the other hand, for any $\e>0$, by the definition of $\widetilde V(0,\mu)$, there exists $\a^\e \in \cA_0$ such that 
\beaa
V\big(t,  \cL_{X_{t}^{0, \xi, \a^\e}}\big)+  \int_0^{t} \dbE[f(s, X_s^{0, \xi, \a^\e}, \a^\e_s)] ds \ge \widetilde V(0,\mu) - \frac{\e}{2}.
\eeaa
Moreover, by the definition of $V\big(t,  \cL_{X_{t}^{0, \xi, \a^\e}}\big)$ there exists $\widetilde \a^\e\in \cA_{t}$ such that
\beaa
J(t,  X_{t}^{0, \xi, \a^\e}, \widetilde \a^\e)  \ge V\big(t,  \cL_{X_{t}^{0, \xi, \a^\e}}\big) - \frac{\e}{2}.
\eeaa
Note that $\hat \a^\e :=  \a\1_{[0, t)} + \a' \1_{[t, T]}  \in \cA_0$. Then, by the middle line of \reff{JDPP},
\beaa
V(0,\mu) &\ge& J(0, \xi, \hat \a^\e)  = J(t,  X_{t}^{0, \xi, \a^\e}, \widetilde \a^\e) +  \int_0^{t} \dbE[f(s, X_s^{0, \xi, \a^\e}, \a^\e_s)] ds\\
&\ge& V\big(t,  \cL_{X_{t}^{0, \xi, \a^\e}}\big) +  \int_0^{t} \dbE[f(s, X_s^{0, \xi, \a^\e}, \a^\e_s)] ds - \frac{\e}{2} \ge  \widetilde V(0,\mu) - \e.
\eeaa
Because $\e>0$ is arbitrary, we obtain $V(0,\mu) \ge \widetilde V(0,\mu)$.
\end{proof}

\begin{remark}
\label{rem-DPP}
Since we are in the simple setting of deterministic control, no regularity or even measurability of $V$ in terms of $(t, \mu)$ is needed in the above result.
\end{remark}

\subsection{The master equation}
In this subsection we derive the master equation associated with the value function $V$.  For this purpose, we first introduce the $2$-Wasserstein distance on $ \cP_2(\dbR^d)$:  for $\mu, \mu'\in \cP_2(\dbR^d)$, 
\bea
\label{W2}
\cW_2(\mu,\mu') := \inf\big\{ \|\xi - \xi'\|_{\dbL^2}:  \xi\in \dbL^2_\mu(\cF_T),  \xi' \in \dbL^2_\mu(\cF_T)\big\}.
\eea 
Let $V: \Th\to \dbR$. The time derivative of $V$ is defined in the standard way:
\bea
\label{patV}
\pa_t V(t,\mu) := \lim_{\d \downarrow 0} {V(t+\d, \mu) - V(t,\mu) \over \d},
\eea
provided the limit exists. Notice that the above is actually the right time derivative. The derivative in terms of $\mu$ is much more involved. We first lift the function $V$:
\bea
\label{UV}
U(t, \xi) := V(t, \cL_\xi),\q  \xi\in \dbL^2(\cF_t).
\eea
 Assume $U$ is continuously  Fr\'{e}chet differentiable in $\xi$, then the Fr\'{e}chet derivative $D U(t,\xi)$ can be identified as an element in $\dbL^2(\cF_t)$.  By  \cite{Cardaliaguet} (based on Lions' lecture), there exists a deterministic function $\pa_\mu V:\Th \times \dbR^d\to \dbR^d$ such that $D U(t, \xi) =  \pa_\mu V(t, \mu, \xi)$. See also \cite{WZ1} for an elementary proof.  This function $\pa_\mu V$ is our spatial derivative, which is called $L$-derivative or  Wasserstein gradient. In particular, the $L$-derivative is also a Gat\^eux derivative:
\bea
\label{pamuV}
\dbE\Big[\pa_\mu V(t, \mu, \xi) \cd \xi'\Big] = \lim_{\e\to 0} {V(t, \cL_{\xi + \e \xi'}) - V(t, \mu)\over \e},
\eea
for all $\xi\in \dbL^2_\mu(\cF_t)$ and $\xi'\in \dbL^2(\cF_t)$.

\begin{remark}
\label{rem-unique}
We shall remark that $\pa_\mu V(t,\mu, \cd): \dbR^d\to \dbR$ is unique only in the support of $\mu$.  Assume $\pa_\mu V$ exists and can be extended to $\dbR^d$ continuously, then we may define $\pa_x \pa_\mu V$ as the standard derivative of $\pa_\mu V$ in terms of the third variable. Obviously, $\pa_x\pa_\mu V(t,\mu, \cd)$ is also well defined only in the support of $\mu$. In this paper we shall always understand $\pa_\mu V$ in this way. In particular, we emphasize that the possible non-uniqueness of $\pa_\mu V(t,\mu, \cd)$ outside of the support of $\mu$ does not affect the It\^o formula \reff{Ito} below, which is what we will actually need in the paper. 
\end{remark}

%\begin{definition}
%\label{defn-C12}
%We say $V\in C^{1,2}(\Th)$ if $\pa_t V, \pa_\mu V, \pa_x\pa_\mu V$ exist, and $V$ as well as these derivatives are uniformly continuous in all their variables. 
%\end{definition}
\begin{definition}
\label{defn-C12}
(i) Let $C^1_{Lip,b}(\cP_2(\dbR^d))$ denote the space of functions $f : \cP_2(\dbR^d)$ $\to \dbR$ such that $\pa_\mu f$ exists everywhere and $\pa_\mu f:\cP_2(\dbR^d) \times \dbR^d \to \dbR^d$ is bounded and Lipschitz continuous.
%, meaning there exists $C > 0$ such that 
%\begin{enumerate}[(i)]
%
%\item $|\pa_\mu V(t, \mu, x)| \leq C$, for all $\mu \in \cP_2(\dbR^d)$ and $x \in \dbR^d$;
%
%\item $|\pa_\mu V(t, \mu, x) - \pa_\mu V(t, \mu', x')| \leq C (\cW_2(\mu, \mu') + |x - x'|)$, for all $\mu, \mu' \in \cP_2(\dbR^d)$ and $x, x' \in \dbR^d$;
%
%\end{enumerate}

(ii) Let $C^2_{Lip,b}(\cP_2(\dbR^d))$ denote the subset of $C^1_{Lip,b}(\cP_2(\dbR^d))$ such that 

$\bullet$  For each $x\in \dbR$, all components of  $\pa_\mu f(\cdot, x)$ belongs to $C^1_{Lip,b}(\cP_2(\dbR^d))$; 

$\bullet$ $\pa_\mu^2 f:\cP_2(\dbR^d) \times \dbR^d \times \dbR^d \to \dbR^{d\times d}$ is bounded and Lipschitz continuous;

$\bullet$ $\pa_x \pa_\mu f: \cP_2(\dbR^d) \times \dbR^d \to \dbR^{d\times d}$ exists and it is bounded and Lipschitz continuous.

\ms
(iii)  Let $ C^{1,2}(\Th) := C^{1,2}_{Lip,b}(\Th)$ denote the space of $V: \Th \to \dbR$ such that

$\bullet$ $V(\cdot, \mu) \in C^1([0,T])$, for any $\mu \in \cP_2(\dbR^d)$;

$\bullet$ $V(t, \cdot) \in C^2_{Lip,b}(\cP_2(\dbR^d))$, for any $t \in [0,T]$.
\end{definition}

The following It\^o formula is crucial for the results developed here, see e.g.  \cite{mean_field_associated_pde}  and  \cite{CCD}.
\begin{lemma}
\label{lem-Ito}
Let $V\in C^{1,2}(\Th)$ and $d X_t = b_t dt + \si_t dW_t$, for some $\dbF$-progressively measurable preocesses $b$ and $\si$ such that $\dbE[\int_0^T [|b_t|^2 + |\si_t|^4]dt] <\infty$. Then
\bea
\label{Ito}
{d\over dt} V(t, \cL_{X_t}) = \pa_t V(t, \cL_{X_t}) + \dbE\Big[ \big[b_t \cd \pa_\mu V + {1\over 2}\si\si^\top_t : \pa_x \pa_\mu V\big](t, \cL_{X_t}, X_t)\Big]. 
\eea
\end{lemma}

The main result of this section is the following verification theorem.

\begin{theorem}
\label{thm-value}
Let Assumption \ref{assum-deterministic1} hold and $V\in C^{1,2}(\Th)$. Then $V$ is the value function defined by \reff{Vt} if and only if $V$ is a classical solution to the  master equation:
\bea
\label{master}
 \pa_t V(t, \mu) +  H(t, \mu, \pa_\mu V, \pa_x \pa_\mu V) =0,\q V(T, \mu) = \dbE[g(\xi)],
\eea
where, for $p:\Th \times \dbR^d\to \dbR^d$ and $q: \Th \times \dbR^d \to \dbR^{d\times d}$,
\bea
\label{Hamiltonian}
%\left.\ba{c}
&\dis H(t, \mu, p, q):= \sup_{a\in A} h(t, \mu, p, q, a),&\\
&\dis h(t, \mu, p, q, a) := \dbE\Big[b(t, \xi, a) \cd p(t, \mu, \xi) +  \dis {1\over 2} \si\si^\top(t, \xi, a) :  q(t,\mu, \xi) + f(t, \xi, a)\Big].&\nonumber
% \ea\right.
\eea
Consequently, the above  master equation has at most one  classical solution in $C^{1,2}(\Th)$. 
\end{theorem}

\begin{proof}
We first assume $V\in C^{1,2}(\Th)$ is defined  by \reff{Vt}.  Then clearly $V$ satisfies the terminal condition in \reff{master}. Now fix $(t,\mu)\in \Th$ and $\xi \in \dbL^2_\mu(\cF_t)$. Recall \reff{Hamiltonian} and apply It\^o formula \reff{Ito} on \reff{DPP1} with $t_1 = t, t_2=t+\d$, we have
\bea
\label{DPPIto}
\sup_{\a\in \cA_t} \int_t^{t+\d} \big[\pa_t V(s, \cL_{X_s^{t,\xi,\a}})+ h(s, \cL_{X_s^{t,\xi,\a}}, \pa_\mu V, \pa_x\pa_\mu V, \a_s)\big] ds &=& 0.
\eea
%\yfs{where the derivatives $\pa_\mu V$ and $\pa_x\pa_\mu V$ are computed at $(s, \cL_{X_s^{t,\xi,\a}}, X_s^{t,\xi,\a})$.} 
Under Assumption \ref{assum-deterministic1},  $\cW_2(\cL_{X^{t,\xi,\a}_s}, \mu)\le \|X_s^{t,\xi,\a} - \xi\|_{\dbL^2} \le C\sqrt{\d}$, for $s\in [t, t+\d]$, where $C$ may depend on $\|\xi\|_{\dbL^2}$. By the required regularity on $V$, there exists a modulus of continuity function $\rho$ such that, again for $\a\in \cA_t$ and $s\in [t, t+\d]$, 
\bea
\label{Verror}
\left.\ba{lll}
&&\dis |\pa_t V(s, \cL_{X_s^{t,\xi,\a}}) - \pa_t V(t,\mu)| + |\pa_\mu V(s, \cL_{X_s^{t,\xi,\a}}, X_s^{t,\xi,\a}) - \pa_\mu V(t,\mu, \xi)| \\
&&\dis \qq + |\pa_x\pa_\mu V(s, \cL_{X_s^{t,\xi,\a}}, X_s^{t,\xi,\a}) - \pa_x\pa_\mu V(t,\mu, \xi)| \le \rho\big(C\sqrt{\d} + |X_s^{t,\xi,\a} - \xi|\big),\ms \\ 
&&\dis |b(s, X_s^{t,\xi,\a}, \a_s) - b(t, \xi, \a_s)| + |\si(s, X_s^{t,\xi,\a}, \a_s) - \si(t, \xi, \a_s)|\\
&&\dis \qq+ |f(s, X_s^{t,\xi,\a}, \a_s) - f(t, \xi, \a_s)| \le \rho\big(\d+ |X_s^{t,\xi,\a} - \xi|\big).
\ea\right.
\eea
These lead to, for a possibly different modulus of continuity function $\rho'$, 
\bea
\label{herror}
\big|h(s, \cL_{X_s^{t,\xi,\a}}, \pa_\mu V, \pa_x\pa_\mu V, \a_s) - h(t, \mu, \pa_\mu V, \pa_x\pa_\mu V, \a_s)\big|\le \rho'(\d).
\eea
Then, by \reff{DPPIto} we have, when $\d\to 0$,
\beaa
\pa_t V(t, \mu)+ \sup_{\a\in \cA_t} {1\over \d} \int_t^{t+\d} h(t, \mu, \pa_\mu V, \pa_x\pa_\mu V, \a_s) ds &=& o(1).
\eeaa
On one hand, this clearly implies $\pa_t V(t, \mu)+ H(t, \mu, \pa_\mu V, \pa_x\pa_\mu V) \ge 0$. On the other hand, by restricting the above $\a$ to constant functions we obtain $\pa_t V(t, \mu)+ H(t, \mu, \pa_\mu V, \pa_x\pa_\mu V) \le 0$. That is, $V$ satisfies \reff{master}. 

We now assume $V\in C^{1,2}(\Th)$  is a classical solution of \reff{master}, and want to verify \reff{Vt}. Fix $(t,\mu)\in \Th$ and $\xi\in \dbL^2_\mu(\cF_t)$. For any $\a\in \cA_t$, by It\^o formula \reff{Ito} we have
\bea
\label{JV}
J(t,\xi,\a) &=& \dbE[g(X^{t,\xi,\a}_T)] +  \int_t^T \dbE\big[f(s, X_s^{t,\xi,\a},  \a_s)\big] ds\nonumber \\
&=& V(T, \cL_{X^{t,\xi,\a}_T}) +  \int_t^T \dbE\big[f(s, X_s^{t,\xi,\a},  \a_s)\big] ds\nonumber \\
&=& V(t, \mu)+ \int_t^T \big[\pa_t V(s, \cL_{X_s^{t,\xi,\a}})+ h(s, \cL_{X_s^{t,\xi,\a}}, \pa_\mu V, \pa_x\pa_\mu V, \a_s)\big] ds\nonumber\\
&\le&V(t, \mu)+ \int_t^T \big[\pa_t V(s, \cL_{X_s^{t,\xi,\a}})+ H(s, \cL_{X_s^{t,\xi,\a}}, \pa_\mu V, \pa_x\pa_\mu V)\big] ds\nonumber\\
&=& V(t, \mu).
\eea
On the other hand, fix $\e>0$ and $n\ge 1$, and denote $t_i := t + {i\over n}[T-t]$, $i=0,\cds, n$. We construct an $\a^{n,\e}\in \cA_t$ as follows. First, there exists $a^\e_0 \in A$ such that 
\beaa
h(t_0, \mu, \pa_\mu V, \pa_x\pa_\mu V, a^\e_0) \ge H(t_0, \mu, \pa_\mu V, \pa_x\pa_\mu V) - \frac{\e}{T-t}.
\eeaa
 Define $\a^{n,\e}_s := a^\e_0$ for $s\in [t_0, t_1)$. Next, there exists $a^\e_1\in A$ such that 
 \beaa
 h(t_1, \cL_{X^{t,\xi, \a^{n,\e}}_{t_1}}, \pa_\mu V, \pa_x\pa_\mu V, a^\e_1) \ge H(t_1, \cL_{X^{t,\xi, \a^{n,\e}}_{t_1}}, \pa_\mu V, \pa_x\pa_\mu V) - \frac{\e}{T-t}.
 \eeaa
  Define $\a^{n,\e}_s := a^\e_1$ for $s\in [t_1, t_2)$. Repeat the procedure and define $\a^{n,\e}_s$ for $s\in [t_i, t_{i+1})$ for $i=1,\cds, n-1$.  Clearly  $\a^{n,\e}\in \cA_t$.  Now, by the second equality of \reff{JV} and then by  \reff{Verror} and \reff{herror}, as $n\to \infty$, we have
\beaa
&&J(t,\xi,\a^{n,\e}) - V(t, \mu)\\
&=& \sum_{i=0}^{n-1} \int_{t_i}^{t_{i+1}} \big[\pa_t V(s, \cL_{X_s^{t,\xi,\a^{n,\e}}})+ h(s, \cL_{X_s^{t,\xi,\a^{n,\e}}}, \pa_\mu V, \pa_x\pa_\mu V, \a^{n,\e}_s)\big] ds\\
&=&  \sum_{i=0}^{n-1} \int_{t_i}^{t_{i+1}} \big[\pa_t V(t_i, \cL_{X_{t_i}^{t,\xi,\a^{n,\e}}})+ h(t_i, \cL_{X_{t_i}^{t,\xi,\a^{n,\e}}}, \pa_\mu V, \pa_x\pa_\mu V, \a^{n,\e}_s)\big] ds + o(1)\\
&\ge&  \sum_{i=0}^{n-1} \int_{t_i}^{t_{i+1}} \big[\pa_t V(t_i, \cL_{X_{t_i}^{t,\xi,\a^{n,\e}}})+ H(t_i, \cL_{X_{t_i}^{t,\xi,\a^{n,\e}}}, \pa_\mu V, \pa_x\pa_\mu V) -\e\big] ds + o(1)\\
&=& o(1) - \e.
\eeaa
Here the $o(1)$ may depend on $\|\xi\|_{\dbL^2}$ and we have used the fact that 
$$\sup_{t\le s\le T} \|X^{t,\xi,\a^{n,\e}}_s\|_{\dbL^2} \le C[1+\|\xi\|_{\dbL^2}].$$ 
Sending $n\to \infty$, we see that 
\beaa
\sup_{\a\in \cA_t} J(t,\xi,\a) \ge V(t,\mu) -  \e.
\eeaa
By the arbitrariness of $\e$, we obtain the desired inequality, and hence $V$ is indeed the value function defined by \reff{Vt}.
\end{proof}

\begin{remark}
\label{rem-viscosity}
While we shall provide some positive results in the next section, in general it is difficult to expect classical solutions for nonlinear master equations. There have been some studies on viscosity solutions to such master equations. For example, \cite{PW} proposed a notion of viscosity solution by first lifting the function $V$ to $U$ in the sense of \reff{UV} and then studing the viscosity property of $U$ in the Hilbert space $\dbL^2(\cF_T)$. More recently, \cite{WZ} proposed an intrinsic notion of viscosity solutions in the Wasserstein space directly, which, in particular, is consistent with the classical solution in Theorem \ref{thm-value} when $V$ is smooth.  
\end{remark}

\subsection{The optimal control}
We now turn to the optimal control. 

\begin{theorem}
\label{thm-control}
Let Assumption \ref{assum-deterministic1} hold and $V\in C^{1,2}(\Th)$ be the classical solution to the master equation \reff{master}-\reff{Hamiltonian}.
   Assume further that

\begin{enumerate}[(i)]

\item the Hamiltonian $H(t, \mu, \pa_\mu V, \pa_x\pa_\mu V)$ defined by \reff{Hamiltonian} has an optimal control $a^* = I(t,\mu)\in A$, for any $(t,\mu)\in \Th$, where $I: [0, T]\times \cP_2(\dbR^d) \to A$ is measurable;

\item for a fixed $(t,\mu)\in \Th$ and $\xi\in \dbL^2_\mu(\cF_t)$, the McKean-Vlasov SDE,
\bea
\label{MV}
X^*_s = \xi + \int_t^s b\big(r, X^*_r, I(r, \cL_{X^*_r})\big) dr +  \int_t^s \si\big(r, X^*_r, I(r, \cL_{X^*_r})\big) dW_r.
\eea
\no has a (strong) solution $X^*$;

%\item $\a^* \in \cA_t$, where $\a^*_s := I(s, \cL_{X^*_s})$, $s\in [t, T]$.  

\end{enumerate}

\no Then $\a^*_s := I(s, \cL_{X^*_s})$, $s\in [t, T]$,  is an optimal control for the optimization problem \reff{Vt} with this fixed $(t,\mu)$. 
\end{theorem}
\begin{proof} Note that $X^* = X^{t, \xi, \a^*}$. Set $\a=\a^*$ in \reff{JV}. By optimality condition $(i)$ we see that equality holds for \reff{JV}, namely $J(t, \xi, \a^*) = V(t,\mu)$, implying  that $\a^*$ is optimal. 
\end{proof}

 As in standard control theory, in general, the existence of  the classical solution $V$ is not sufficient for the existence of the optimal control. In particular, the McKean-Vlasov SDE \reff{MV} may not have a solution, even if $I$ exists. At  below we provide a sufficient condition.

\begin{theorem}
\label{thm-control-sufficient}
Let all the conditions in Theorem  \ref{thm-control} hold true, except possibly the (ii) there. Assume further $b, \si$ are bounded and continuous in $a$,  and $I: \Th \to A$ is continuous.
Then the  McKean-Vlasov SDE \reff{MV} has a strong solution for any $(t,\mu)$, and hence  the optimization problem \reff{Vt} has an optimal control.
\end{theorem}
\begin{proof}  Without loss of generality, we prove the result only at $(0, \mu)$. Fix $\xi \in \dbL^2_0(\mu)$.
For any $\a\in \cA_0$, denote
\beaa
X^\a_t = \xi + \int_0^t b(s, X^\a_s, \a_s) ds + \int_0^t \si(s, X^\a_s, \a_s) dW_s.
\eeaa
Under Assumption \ref{assum-deterministic1}, it is clear that 
\bea
\label{cLXcont}
\dbE[|X^\a_t - X^\a_s|^2] \le C_\mu |t-s|,
\q\mbox{and thus}\q \cW_2(\cL_{X^\a_t}, \cL_{X^\a_s}) \le C_\mu \sqrt{|t-s|},
\eea
where the constant $C_\mu$ may depend on $\mu$, but does not depend on $\a$. Moreover, assume $|b|, |\si|\le L$.  Let $\cD_L(\mu)$ denote the set of $\cL_{\tilde X_t}$, where  $t\in [0, T]$, $\tilde X_t = \tilde X_0 + \int_0^t \tilde b_s ds + \int_0^t \tilde\si_s dW_s$ in some arbitrary probability space with $\cL_{\tilde X_0} = \mu$ and  $|\tilde b|, |\tilde \si|\le L$. As in \cite{WZ} Lemma 3.1, one can easily show that $\cD_L(\mu)$ is compact under $\cW_2$. Since $I$ is continuous in $\Th$, then it is uniformly continuous on $[0, T]\times \cD_L(\mu)$ with certain  modulus of continuity function $\rho_\mu$, which may depend on $\mu$. Clearly $\cL_{X^\a_t} \in \cD_L(\mu)$ for all $\a\in \cA_0$ and $t\in [0, T]$. Then we have
\bea
\label{Icont}
\Big|I(t, \cL_{X^\a_t}) - I(s, \cL_{X^\a_s}) \Big|   \le \rho_\mu(|t-s|),\q\mbox{for all}~\a\in \cA_0.
\eea
Denote 
\bea
\label{cArho}
\cA_0(\rho_\mu) := \Big\{\a\in \cA_0: |\a_t - \a_s|\le  \rho_\mu(t-s),\q 0\le s<t\le T\Big\}.
\eea
We now define a mapping  $\Phi: \cA_0(\rho_\mu)\to \cA_0(\rho_\mu)$ by $\Phi_t(\a) := I(t, \cL_{X^\a_t})$, where \reff{Icont} ensures that $\Phi(\a) \in  \cA_0(\rho_\mu)$ for all $\a\in  \cA_0(\rho_\mu)$.  One can easily show  that $ \cA_0(\rho_\mu)$ is convex and compact under the uniform norm, and $\Phi$ is continuous. Then, applying the Schauder's fixed point theorem, $\Phi$ has a fixed point $\a^*\in \cA_0(\rho_\mu)$: $\Phi(\a^*) = \a^*$. Now it is clear that $X^* := X^{\a^*}$ satisfies \reff{MV}, and hence $\a^*$ is  an optimal control. 
\end{proof}

\begin{remark}
\label{rem-SMP}
{\rm In this section, we used the dynamic programming principle. Since the control here is deterministic and thus falls in strong formulation, one may also use the stochastic maximum principle, provided the optimal control exists. We will present heuristic arguments in Appendix \ref{sect-appendix} to show how the McKean-Vlasov SDEs come to play naturally.
}
\end{remark}

\section{Classical solution of a nonlinear master equation}
\label{sect-example}
\setcounter{equation}{0}

 The existence of classical solutions for nonlinear master equations is a very challenging problem. We shall leave the general case to future research. In this section we study a special type of master equations. Consider the equation \reff{master}-\reff{Hamiltonian} with
\beaa
 \si = I_d,\q b = b(t, a), \q f=f(t,a).
\eeaa
Then \reff{master} becomes:
\beaa
\pa_t V (t,\mu) + {1\over 2} \dbE\Big[\tr(\pa_x \pa_\mu V(t, \mu, \xi))\Big] + \sup_a \Big[ b(t,a)\cd \dbE\big[\pa_\mu V(t, \mu, \xi)\big] + f(t,a)\big]=0.
\eeaa
This is a special case of the following nonlinear master equation:
\bea
\label{master2}
\left.\ba{c}
\dis \pa_t V (t,\mu) + {1\over 2} \dbE\Big[\tr(\pa_x \pa_\mu V(t, \mu, \xi))\Big] + F\big(t,  \dbE\big[\pa_\mu V(t, \mu, \xi)\big]\big)=0,\\
  V(T,\mu) = \dbE\big[g(\xi)\big].
\ea\right.
\eea

\begin{theorem}
\label{thm-classical}
Let  $F$ and $g$ be smooth enough with bounded derivatives. Assume one of the following two conditions hold true:

(i) $T$ is sufficiently small;

(ii) $d=1$; and  either $\pa_{xx} g >0>\pa_{xx} F$, or $\pa_{xx} g < 0<\pa_{xx} F$.

\no Then the master equation \reff{master2} has a classical solution $V\in C^{1,2}(\Th)$.
\end{theorem}
\begin{proof} We shall proceed in two steps.

{\it Step 1. } % Let $\pa_x F$ and $\pa_x g$ denote the standard spatial derivatives of $F$ and $g$. 
Consider the following master equation which is linear in $\pa_\mu \widetilde V$:
\bea
\label{master2-1}
\left.\ba{c}
\dis \pa_t \widetilde V(t,\mu) + {1\over 2} \dbE\Big[\tr(\pa_x \pa_\mu \widetilde V(t, \mu, \xi))\Big] + \pa_x F\big(t, \widetilde V(t,\mu)\big)   \dbE\big[\pa_\mu \widetilde V(t, \mu, \xi)\big]=0,\\
 \widetilde V(T,\mu) = \dbE\big[\pa_xg(\xi)\big].
\ea\right.
\eea
We shall prove in Step 2 below that, under (i) or (ii) the above master equation has a unique classical solution $\widetilde V$. We next consider the linear master equation:
\bea
\label{master2-2}
\left.\ba{c}
\dis \pa_t V (t,\mu) + {1\over 2} \dbE\Big[\tr(\pa_x \pa_\mu V(t, \mu, \xi))\Big] + F\big(t,  \widetilde V(t, \mu)\big)=0,\\
V(T,\mu) = \dbE\big[g(\xi)\big],
\ea\right.
\eea
Then clearly $V$ is also smooth. It remains to verify that the above $V$ satisfies \reff{master2}. Indeed, by \reff{master2-2} we see that, 
\bea
\label{linearrep}
V(t,\mu) = \dbE\big[g(X^{t,\xi}_T)\big] + \int_t^T F(s, \widetilde V(s, \cL_{X^{t,\xi}_s})) ds,
\eea
where $X^{t,\xi}_s := \xi + W_s-W_t$. Differentiating with respect to $\mu$, we obtain
\bea
\label{lineardiff}
\dbE\big[\pa_\mu V(t, \mu,\xi)\big] &=& \dbE\Big[\pa_x g(X^{t,\xi}_T)\Big] \\
&+& \dbE\Big[\int_t^T \pa_x F(s, \widetilde V(s, \cL_{X^{t,\xi}_s}))  \cd \pa_\mu\widetilde V\big(s, \cL_{X^{t,\xi}_s}, X^{t,\xi}_s\big) ds\Big]. \nonumber
\eea
That is, $\ol V(t,\mu) := \dbE\big[\pa_\mu V(t, \mu,\xi)\big] $ satisfies the following linear master equation:
\beaa
\left.\ba{c}
\dis \pa_t \ol V(t,\mu) + {1\over 2} \dbE\Big[\tr(\pa_x \pa_\mu \ol V(t, \mu, \xi))\Big] + \pa_x F\big(t, \widetilde V(t,\mu)) \cd  \dbE\big[\pa_\mu \widetilde V(t, \mu, \xi)\big]\big)=0,\\
 \ol V(T,\mu) = \dbE\big[\pa_xg(\xi)\big].
\ea\right.
\eeaa
However, by  \reff{master2-1}, $\widetilde V$ also satisfies the above master equation. Then by the uniqueness of classical solutions, we have  $\widetilde V(t,\mu) =\ol V(t,\mu) =  \dbE\big[\pa_\mu V(t, \mu,\xi)\big]$. Plugging this into \reff{master2-2}, we see that $V$ satisfies \reff{master2}.

\ms

{\it Step 2.} We now prove the wellposedness of \reff{master2-1} under (i) or (ii). When $T$ is small, the arguments are rather standard, see e.g. \cite{CCD}. We now assume (ii) holds true. Without loss of generality, we assume $F$ is convex in $x$ and $g$ is concave.  For any $y\in \dbR$, define
\bea
\label{Phi}
\left.\ba{c}
\dis \Phi(y; t, \mu) := \dbE\Big[\pa_x g\big(\xi +  W_{T}-W_t+ \int_t^T \pa_x F(s,y)ds\big)\Big],\\
\dis \Psi(y, t,\mu) := \Phi(y; t,\mu)-y,
\ea\right.
\eea
where $\cL_\xi = \mu$ and $W_T-W_t$ is independent of $\xi$.  It is straightforward to show  that $\Phi$ is smooth in $(y, t,  \mu)$ and, for any $y$, $\Phi(y; \cd)$ solves the following linear master equation:
\bea
\label{nonlinearMaster3}
 \pa_t \Phi (y; t, \mu) + {1\over 2} \dbE\big[\pa_x \pa_\mu \Phi(y; t,  \mu, \xi))\big] + \pa_x F(t, y) \dbE\big[\pa_\mu \Phi(y; t,\mu, \xi)\big]=0.
 \eea
Under our conditions, $\pa_x g$ is decreasing and $\pa_x F$ is increasing in $y$, then by \reff{Phi} $\Phi$ is decreasing in $y$ and thus  $\pa_y \Psi\le -1$, so $y \mapsto \Psi(y, t, \mu)$ has an inverse function $\Psi^{-1}$, which is also smooth.  Since $\pa_x g$ is bounded by some constant $C_0$, then $|\Phi(y; t, \mu)|\le C_0$, and thus $\Psi(C_0, t, \mu) \le 0 \le \Psi(-C_0, t, \mu)$ for any fixed $(t,\mu)$. In particular, $0$ is in the range of $\Psi(\cd; t,\mu)$ for any fixed $(t,\mu)$. Define $U(t,\mu) := \Psi^{-1}(0, t,\mu)$, then $U$ is smooth. 
Note that $U(t, \mu) = \Phi(U(t,\mu); t,  \mu)$. Apply the chain rule (which is obvious from the definitions), we have
 \beaa
 \pa_t U = \pa_t \Phi + \pa_y \Phi \pa_t U,\q \pa_\mu U = \pa_\mu \Phi + \pa_y \Phi \pa_\mu U,\q \pa_x\pa_\mu U = \pa_x\pa_\mu \Phi + \pa_y \Phi \pa_x\pa_\mu U.
 \eeaa
 Namely, denoting $c := 1-\pa_y \Phi(U(t,\mu); t, \mu) \ge 1$,
 \beaa
 & \pa_t \Phi(U(t,\mu); t, \mu) = c~ \pa_t U(t,\mu), \q \pa_\mu \Phi(U(t,\mu); t, \mu,\cd) = c~ \pa_\mu U(t,\mu,\cd), &\\
&    \pa_x\pa_\mu \Phi(U(t,\mu); t, \mu,\cd) = c~ \pa_x\pa_\mu U(t,\mu,\cd).&
  \eeaa
Plug these into \reff{nonlinearMaster3} with $y=U(t,\mu)$, we obtain that $U$ satisfies \reff{master2-1}.
\end{proof}

%\begin{remark}
%\label{rem-classical}
%{\rm The condition of small $T$ is used only to ensure the solvability of the master equation \reff{master2-2}, which corresponds to a McKean-Vlasov type of coupled forward backward SDE. As mentioned in \cite{CCD} and similar to the literature of the standard FBSDE, there are other sufficient conditions for \reff{master2-2}. Then under those conditions, \reff{master2} also admits a classical solution. 
%}
%\end{remark}

\subsection{An example}
\label{sect-eg}
We now consider a special case. For some $R>0$, which will be specified later, set
\bea
\label{example}
d=1,\q A = [-R, R],\q b(t, x, a) = a,\q \si = 1, \q f(t,x, a) = -{1\over 2} a^2.
\eea
Then
\bea
\label{eg-H}
\left.\ba{c}
\dis h(t,\mu, p, q, a) = {1\over 2} \dbE[q(t,\mu,\xi)] + a \dbE[p(t,\mu,\xi)] - {1\over 2} a^2,\\
\dis H(t,\mu, p,q) =  {1\over 2} \dbE[q(t,\mu,\xi)] + F(\dbE[p(t,\mu,\xi)]),\\
\dis \mbox{where}~ F(x) = {1\over 2}|x|^2\1_{\{|x|\le R\}} + [R|x| - {1\over 2} R^2 ] \1_{\{|x|> R\}}.
\ea\right.
\eea
and thus \reff{master} becomes:
\bea
\label{master-quadratic}
\pa_t V +   {1\over 2} \dbE\big[\pa_x \pa_\mu V(t,\mu,\xi)\big] + F\big(\dbE\big[\pa_\mu V(t,\mu,\xi)\big) = 0,\q V(T,\mu) = \dbE[g(\xi)].
\eea
Notice that $F$ is convex,   however, it  is in $C^1(\dbR)$ but not in $C^2(\dbR)$.

\begin{theorem}
\label{thm-eg}
Assume $g$ is smooth  enough with bounded derivatives, and in particular $|\pa_x g| \le C_0 < R$. Then,   either for $T$ small enough, or $d=1$ and $g$ is concave, 

(i) the master equation \reff{master-quadratic} has a unique classical solution $V$ such that 
\bea
\label{VLipschitz}
\Big|\dbE\big[\pa_\mu V(t,\mu,\xi)\big]\Big|\le C_0,\q (t,\mu) \in \Th, \xi \in \dbL^2_\mu(\cF_t);
\eea
%where $C_0$ is an upper bound of $\pa_x g$;

(ii) for any $(t,\mu) \in \Th$ and $\xi \in \dbL^2_\mu(\cF_t)$, the McKean-Vlasov SDE \reff{MV}  with $I(t, \mu) :=  \dbE\big[\pa_\mu V(t,\mu,\xi)\big]$ has a solution $X^*$; 

(iii)  for any $(t,\mu) \in \Th$, the optimization problem \reff{Vt} has an optimal control: $\a^*_s := I(s, \cL_{X^*_s})$. 
\end{theorem}
\begin{proof}
Let $\widetilde F: \dbR \to \dbR$ be a smooth function such that
\beaa
\mbox{$\widetilde F$ is convex and $\widetilde F(x) =  F(x)$ for $|x|\le C_0$ or   $|x|\ge R$.}
\eeaa
Applying Theorem \ref{thm-classical}, the master equation \reff{master-quadratic} corresponding to $\widetilde F$ has a classical solution $V$.  Introduce the conjugate of $\widetilde F$:
$\widetilde f(a) := \sup_{x\in \dbR} [a x -  \widetilde F(x)]$,  $a\in A$. By the convexity of $\widetilde F$, we have $\widetilde F(x) = \sup_{a\in A} [a x - \widetilde f(a)]$.
Then by Theorem \ref{thm-value} we see that
\bea
\label{VtildeJ}
&\dis V(t,\mu) = \sup_{\a\in  \cA_t} \widetilde J(t, \xi, \a),\q\mbox{where}&\\
&\dis \widetilde X^{t,\xi,\a}_s := \xi + \int_t^s \a_r dr + W_s-W_t,\q  \widetilde J(t, \xi, \a) := \dbE\big[g(\widetilde X^{t,\xi,\a}_T)\big] - \int_t^T \widetilde f(\a_s) ds.&\nonumber
\eea
For any $t\in [0, T]$, $\xi, \xi'\in \dbL^2(\cF_t)$, and $\a\in  \cA_t$, under our conditions it is clear that, 
\beaa
|\widetilde J(t, \xi, \a) -\widetilde J(t,\xi',\a)| \le C_0\dbE[|\xi-\xi'|]. % \le C_0 \cW_2(\cL_\xi, \cL_{\xi'}).
\eeaa
Since $\xi, \xi'$ are arbitrary, then it follows from  \reff{VtildeJ} that
\beaa
|V(t, \mu) -V(t,\mu')| \le C_0 \cW_2(\mu, \mu'),
\eeaa
which implies \reff{VLipschitz}  immediately. Since $\widetilde F(x) = F(x) = {1\over 2} x^2$ for $|x|\le C_0$, then \reff{VLipschitz} implies further that $V$ is a classical solution to master equation \reff{master-quadratic} corresponding to $F$.

(ii) Clearly in this case  the optimal argument of the Hamiltonian $F$ leads to $I(t,\mu) = \dbE\big[\pa_\mu V(t,\mu,\xi)\big]$, which is continuous. Then (ii) follows from Theorem \ref{thm-control-sufficient}.

Finally, $(iii)$ follows directly from Theorem \ref{thm-control}.
\end{proof}  

We remark that in this example it is more natural to set $A = \dbR$ and all the results still hold true. The constraint $A= [-R, R]$ is to ensure the uniform requirement in Assumption \ref{assum-deterministic1} $(i)$, which is more convenient for establishing the general theory, but can be relaxed.

\section{The General Case}
\label{sect-general}
\setcounter{equation}{0}
In this section we investigate the general case $T>\ch$.

\subsection{Strong formulation with closed loop controls}
In this subsection we illustrate how the information delay naturally lead to the path dependence of the value function, even if  the coefficients $b, \si, f, g$ in \reff{Ja} depend only on the current state of $X$. It is easier to show the idea in strong formulation, namely we fix a probability space and the state process $X^\a$ is controlled,  but we emphasize that we shall use  closed loop controls, both for practical and for theoretical reasons.
% \yfs{Strong formulation means that the controls are chosen to be adapted to $\xi_{\cdot - \ch}$.}

As in Section \ref{sect-deterministic}, let  $(\Om, \dbF, \dbP)$ be a filtered probability space on $[0, T]$, and $W$ an $\dbF$-Brownian motion under $\dbP$.  For simplicity, in this subsection we assume $T \le 2\ch$, which will not be required in later subsections. Let $t\in (\ch, T]$, and $\xi$ be an $\dbF$-progressively measurable process on $[0, t]$. Consider the following counterpart of  \reff{Ja}: 
\bea
 \label{Ja2}
 \left.\ba{c}
\dis X_s^{t, \xi, \a} = \xi_t + \int_t^s b\big(r, X_r^{t, \xi, \a}, \a_r(\xi_{[0, r-\ch]})\big)dr + \int_t^s \sigma\big(r, X_r^{t, \xi, \a},  \a_r(\xi_{[0, r-\ch]})\big) dW_r;\\
\dis J(t,\xi, \a) := \dbE\left[g(X_T^{t, \xi, \a}) + \int_t^T f\big(s, X_s^{t, \xi, \a},  \a_s(\xi_{[0, s-\ch]})\big)ds \right].
\ea\right.
\eea
Similar to Lemma \ref{lem-invariant}, $J(t,\xi, \a)$ depends on  $\xi$ only through the law of the stopped process $\xi_{[0,t]}$. That is, if $\xi'$ is another process  such that $\cL_{\xi_{[0,t]}} = \cL_{\xi'_{[0,t]}}$, then $J(t, \xi, \a) = J(t,\xi', \a)$. Consequently, the following value function is also law invariant:
 \bea
\label{V2}
 \widetilde V(t,\xi) := \sup_{\a\in \cA_t} J(t,\xi, \a).
 \eea
 
We emphasize that the above law invariant property relies on the law of the stopped process $\xi_{[0,t]}$, rather than the law of the current state $\xi_t$.

 \begin{example}
 \label{eg-path}
 Let $d=1$, $A = [-1,1]$, $b(t,x,a) = a$, $\si(t,x,a) = 1$, $f(t,x,a)=0$, $g(x) = x^2$, $T=2\ch$, $t= {3\over 2} \ch$. Set
 \bea
 \label{xixi'}
 \xi_s = W_s, 0\le s\le t,\q \xi'_s := W_{3(s-\ch)}\1_{[\ch, t]}(s).
 \eea
 Then $\xi_t = \xi'_t=W_t$ but  in general $\widetilde V(t,\xi)\neq \widetilde V(t, \xi')$.
  \end{example}
\begin{proof}
First, since $\xi'_s = 0$, $s\le \ch$, then $ \a_r(\xi'_{[0, r-\ch]})=\a_r(0)$ is deterministic. Thus
\beaa
J(t,\xi', \a) = \dbE\Big[\big|W_t + \int_t^T \a_r(0) dr + W_T-W_t\big|^2\Big]=\big| \int_t^T \a_r(0) dr \big|^2 + T.
\eeaa
This implies 
\beaa
\widetilde V(t,\xi') = T + (T-t)^2 =2\ch + {1\over 4} \ch^2.
\eeaa
On the other hand, denote $\b_r := \a_{r+\ch}(W_{[0,r]})$ which is $\cF_r$-measurable, then
\beaa
J(t,\xi, \a) &=& \dbE\Big[\big|W_t + \int_{t-\ch}^{T-\ch} \b_r dr + W_T-W_t\big|^2\Big]\\
 &=& \dbE\Big[\big|\int_{t-\ch}^{T-\ch} \b_r dr|^2  + 2\int_{t-\ch}^{T-\ch} W_r\b_r dr\Big] + T \ge 2 \dbE\Big[\int_{\ch\over 2}^{\ch} W_r\b_r dr\Big] + 2\ch.
\eeaa
By choosing $\b_r = \mbox{sign}(W_r)$, we have
\beaa
\widetilde V(t,\xi) \ge  2 \dbE\Big[\int_{\ch\over 2}^{\ch} |W_r|dr\Big] + 2\ch = 2\ch + c  \ch^{3\over 2},
\eeaa 
where $c>0$ is a generic constant independent of $\ch$. Then clearly $\widetilde V(t,\xi) > \widetilde V(t,\xi')$, when $\ch$ is small enough.
\end{proof}

We also remark that it is crucial to use closed loop controls. If we use open loop controls with delay, namely $\a_s = \a_s(W_{[0, s-\ch]})$, then for each $\a$, obviously $J(t, \xi, \a)$ would depend on the joint law of $(\xi, W)$ on $[0, t]$. The following example shows that the corresponding value function $\tilde V(t,\xi)$ may also violate the law invariant property.
 \begin{example}
 \label{eg-open}
 Consider the same setting in Example \ref{eg-path}, but replace \reff{xixi'} with
 \beaa
 \xi_s = [W_s-W_\ch]\1_{[\ch, t]}(s), \q  \xi'_s = W_{s-\ch}\1_{[\ch, t]}(s),\q 0\le s\le t.
 \eeaa
 Then $\cL_{\xi_{[0, t]}} =\cL_{\xi'_{[0, t]}}$. However, if we use open loop controls but still denote the value function as $\tilde V$, then  $\widetilde V(t,\xi)\neq \widetilde V(t, \xi')$.
  \end{example}
\begin{proof}
First, note that $\a_s = \a_s(W_{[0,s-\ch]})$ is $\cF_\ch$-measurable. Then
\beaa
J(t,\xi, \a) = \dbE\Big[\big|W_t  - W_\ch+ \int_t^T \a_s ds + W_T-W_t\big|^2\Big]=\dbE\Big[\big|\int_t^T \a_s ds\big|^2\Big] + T-\ch.
\eeaa
This implies 
\beaa
\widetilde V(t,\xi) = T -\ch + (T-t)^2 = \ch + {1\over 4} \ch^2.
\eeaa
On the other hand, denote $\b_r := \a_{r+\ch}(W_{[0,r]})$ which is $\cF_r$-measurable, then
\beaa
J(t,\xi', \a) = \dbE\Big[\big|W_{\ch\over 2} + \int_{t-\ch}^{T-\ch} \b_r dr + W_T-W_t\big|^2\Big]= \dbE\Big[\big|W_{\ch\over 2} + \int_{\ch\over 2}^{\ch} \b_r dr\big|^2\Big]+{\ch\over 2}.
\eeaa
By choosing $\b_r = \mbox{sign}(W_{\ch\over 2})$,  we have
\beaa
\widetilde V(t,\xi) \ge \dbE\Big[\big[|W_{\ch\over 2}| + {\ch\over 2}\big]^2\Big] +{\ch\over 2} = \ch + {1\over 4} \ch^2 + \ch \dbE[|W_{\ch\over 2}|] > \widetilde V(t,\xi).
\eeaa 
This completes the proof. 
\end{proof}

\subsection{Weak formulation in  path-dependent setting}

Both for closed loop controls and for path-dependent problems, it is a lot more convenient to use the weak formulation on canonical space. We shall follow the setting of \cite{WZ}.

 Let $\Om := C([0, T]; \dbR^d)$  be the canonical space equipped with the metric $\|\o\|_T := \sup_{0\le t\le T}|\o_t|$, $X$ the canonical process, and $\dbF = \dbF^X = \{\cF_t\}_{t \in [0,T]}$ the natural filtration generated by $X$. Denote by $\cP_2(\cF_T)$ the set of probability measures $\dbP$ on $\cF_T$ such that $\dbE^\dbP[\|X\|^2_T] <\infty$ and $\Th_T := [0, T]\times \cP_2(\cF_T)$. Quite often we will also use $\mu$ to denote elements of $\cP_2(\cF_T)$.   
 %For each $\mu\in \cP_2(\cF_T)$, denote $\dbL^2_\mu(\dbF) := \{\eta \in \dbL^2(\dbF): \cL_\eta = \mu\}$, where $\cL_\eta$ denotes the law of the process $\eta$.  
 We equip $\cP_2(\cF_T)$ with the $2$-Wasserstein distance  $\cW_2$ which extends \reff{W2}: 
 \bea
 \label{W2path}
 \cW_2(\mu,\mu') := \inf\big\{ \big(\dbE[\sup_{0\le t\le T} |\eta_t - \eta'_t|^2\big]\big)^{1\over 2}:  \cL_\eta = \mu, \cL_{\eta'} = \mu' \big\},
\eea 
for $\mu, \mu' \in \cP_2(\cF_T)$, where $\cL_\eta$ is the law of the process $\eta$.  %constraint of $\mu$ on $\cF_t$, and ${\color{red}\mu_{[0,t]}}$ the constraint of $\mu$ on $\si(X_t)$. Finally, by abusing the notations, we denote $\Th := [0, T] \times \cP_2(\cF_T)$ and $\ol \Th := \{(t,\eta): t\in [0, T], \eta \in \dbL^2(\dbF)\}$.

Given $\mu\in \cP_2(\cF_T)$, let $  \mu_{[0,t]}$ denote the $\mu$-distribution of the stopped process $X_{[0, t]}$. For a function $V: \Th_T\to \dbR$, we say $V$ is $\dbF$-adapted if $V(t, \mu) = V(t, \mu_{[0,t]})$ for any $(t, \mu)\in \Th_T$. For such $V$, we define its time derivative as:
\bea
\label{patV2}
\pa_t V(t, \mu) := \lim_{\d\downarrow 0} {V(t+\d,  \mu_{[0,t]}) - V(t,  \mu_{[0,t]}) \over \d},
\eea
where we are freezing the law of $X$ from $t$ to $t+\delta$. The spacial derivative takes the form $\pa_\mu V: \Th_T \times \Om\to \dbR^d$ and is $\dbF$-progressively measurable, namely measurable in all variables and $\dbF$-adapted. We emphasize that, as in \cite{fito_dupire}, $\pa_\mu V$ is not a Fr\'echet derivative with respect to the law of the whole stopped process $X_{[0,t]}$, but a derivative with respect to $\cL_{X_t}$ only. Roughly speaking, by extending the whole setting to the space of {\cad} paths, let $\xi$ be a process on $[0, t]$ such that $\cL_\xi =  \mu_{[0,t]}$, and $\xi'_t$ be an $\cF_t$-measurable random variable. Then 
\bea
\label{pamuV2}
\dbE^\mu\Big[\pa_\mu V(t, \mu, \xi) \cd \xi'_t\Big] := \lim_{\e\downarrow 0} {V(t,  \cL_{\xi + \e \xi'_t\1_{\{t\}}})  - V(t, \cL_\xi) \over \e}.
\eea
Moreover, for the process $\pa_\mu V(t, \mu, X_\cd)$, we may introduce the path derivative $\pa_\o \pa_\mu V$ in the spirit of  \cite{fito_dupire}. When $V$ is smooth enough in these senses, the functional It\^{o} formula \reff{FIto} below holds true. We refer to \cite{WZ} for details.  In this section, to avoid the technical details,  we take the approach in \cite{fito_zhang_viscosity_I} and use the functional It\^o formula directly to define the smoothness of $V$.

\begin{definition}
\label{defn-C12path}
Let $C^{1,2}(\Th_T )$ denote the space of   functions $V: \Th_T \to \dbR$  such that there exist  functions  $\pa_\mu V:  \Th_T \times \Om\to \dbR^d$ and   $\pa_\o\pa_\mu V:  \Th_T \times \Om\to \dbR^{d\times d}$ satisfying:

(i)  the $\pa_t V$ defined by \reff{patV2} exists, and $V$, $\pa_t V$, $\pa_\mu V$, $\pa_\o\pa_\mu V$ are all $\dbF$-adapted and  uniformly continuous;

(ii) for any semimartingale measure $\dbP$, namely $X$ is a semimartingale under $\dbP$, the following functional It\^o formula holds: 
\bea
\label{FIto}
\q d V(t, \dbP) = \pa_t V(t, \dbP) dt + \dbE^\dbP\Big[\pa_\mu V(t, \dbP, X_\cd) \cd dX_t + {1\over 2} \pa_\o \pa_\mu  V(t, \dbP, X_\cd) : d\la X\ra_t\Big]. 
\eea
\normalsize
\end{definition}
By Lemma \ref{lem-Ito} and \cite{WZ}, the spatial derivatives there coincide with the above  $\pa_\mu V, \pa_\o \pa_\mu V$  (with $\pa_\o\pa_\mu V = \pa_x \pa_\mu V$ in Markovian case).  
We remark that, for the purpose of viscosity solutions, in \cite{WZ}, \reff{FIto} is required only for semimartingale measures whose drift and diffusion characteristics are bounded. In that case, the regularity requirements on $V$ are weaker than the corresponding conditions in Lemma \ref{lem-Ito}. It is not difficult to extend the functional It\^o formula in \cite{WZ} to allow for more general semimartingale measures. Nevertheless, it is more convenient to define the derivatives through the functional It\^o formula directly as we do here.

\begin{lemma}
\label{lem-pamuunique}
For any $V\in C^{1,2}(\Th_T )$, the derivatives  $\pa_\mu V$ and $\pa_\o\pa_\mu V$ are unique in the sense that  $\pa_\mu V(t, \mu, X_\cd)$ and  ${1\over 2}[\pa_\o\pa_\mu V+  (\pa_\o\pa_\mu V)^\top](t, \mu, X_\cd)$ are $\mu$-a.s. unique for any $(t,\mu)\in \Th_T$.
\end{lemma}
We remark that, since $\la X\ra$ is symmetric, so the uniqueness of ${1\over 2}[\pa_\o\pa_\mu V+  (\pa_\o\pa_\mu V)^\top]$ $(t, \mu, X_\cd)$ implies that uniqueness of $\pa_\o \pa_\mu  V(t, \dbP, X_\cd) : d\la X\ra_t$ in \reff{FIto}.

\begin{proof} First let $\mu\in \cP_2(\cF_T)$ be a semimartingale measure. For any $t \in [0, T]$ and any $\cF_t$-measurable and  bounded  random variables $b_t$ and $\si_t>0$, let $\dbP\in \cP_2(\cF_T)$ be such that 
\beaa
\dbP_{[0,t]} = \mu_{[0,t]}\q\mbox{and}\q X_s - X_t = b_t[s-t] + \si_t [W_s - W_t],~ t\le s\le T,~ \dbP\mbox{-a.s.},
\eeaa
for some $\dbP$-Brownian motion $W$.  Then, by \reff{FIto}, we see that 
\beaa
\dbE^\dbP\Big[ b_t \cd \int_t^s \pa_\mu V(r,  \dbP, X_\cd)dr + {1\over 2} \si_t\si^\top_t : \int_t^s \pa_\o \pa_\mu V(r,  \dbP, X_\cd)dr\Big]
\eeaa
 is unique. By the uniform continuity of $\pa_\mu V$ and $\pa_\o \pa_\mu V$, this implies that  
 \beaa
 \dbE^{\mu} \Big[ b_t\cd \pa_\mu V(t,  \mu, X_\cd) + {1\over 2}\si_t\si^\top_t : \pa_\o\pa_\mu V(t,  \mu, X_\cd) \Big] 
 \eeaa
  is unique. Here we rewrite $\dbP$ as $\mu$ since $\dbP_{[0,t]} = \mu_{[0,t]}$ and the integrand at above is $\cF_t$-measurable. Since $b_t$ and $\si_t$ are arbitrary, we obtain the desired uniqueness.

Now assume $\mu\in \cP_2(\cF_T)$ is arbitrary. For any $\e>0$, denote $X^\e_t := {1\over \e} \int_{(t-\e)^+}^t X_s ds$ and $\mu^\e := \mu \circ (X^\e)^{-1}$. Then 
\beaa
\cW^2_2(\mu, \mu^\e) \le \dbE^\mu \Big[\|X-X^\e\|_T^2\Big] \to 0,\q\mbox{as}~ \e\to 0,
\eeaa
which implies that $\mu^\e\to \mu$ weakly.
Clearly $X^\e$ is an $\mu$-semimartinagle, then $\mu^\e$ is a semimartingale measure. Thus $\pa_\mu V(t, \mu^\e_t, X_\cd)$ is $\mu^\e$-a.s. unique.  Let $\eta_t$ be $\cF_t$-measurable, bounded, and continuous in $\o$ (under $\|\cd\|_T$). Note that, denoting by $\rho$ the modulus of continuity function of $\pa_\mu V$,
\beaa
&&\Big|\dbE^{\mu_\e} \Big[\pa_\mu V(t, \mu^\e, X_\cd) \eta_t\Big] - \dbE^{\mu} \Big[\pa_\mu V(t, \mu, X_\cd) \eta_t\Big] \Big|\\
&\le& \Big|\dbE^{\mu_\e} \Big[\pa_\mu V(t, \mu^\e, X_\cd) \eta_t\Big] - \dbE^{\mu_\e} \Big[\pa_\mu V(t, \mu, X_\cd) \eta_t\Big] \Big|\\
&&+\Big|\dbE^{\mu_\e} \Big[\pa_\mu V(t, \mu, X_\cd) \eta_t\Big] - \dbE^{\mu} \Big[\pa_\mu V(t, \mu, X_\cd) \eta_t\Big] \Big|\\
&\le& C\rho(\cW_2(\mu, \mu^\e)) + \Big|\dbE^{\mu_\e} \Big[\pa_\mu V(t, \mu , X_\cd) \eta_t\Big] - \dbE^{\mu} \Big[\pa_\mu V(t, \mu , X_\cd) \eta_t\Big] \Big|.
\eeaa
Sending $\e\to 0$, by the weak convergence of $\mu^\e\to \mu$, we see that 
\beaa
 \dbE^{\mu} \Big[\pa_\mu V(t, \mu , X_\cd) \eta_t\Big]  = \lim_{\e\to 0} \dbE^{\mu_\e} \Big[\pa_\mu V(t, \mu^\e , X_\cd) \eta_t\Big]
 \eeaa
is unique. Since $\eta_t$ is arbitrary, we obtain the desired uniqueness of $\pa_\mu V(t, \mu , X_\cd)$. Similarly we have the uniqueness of $\pa_\o\pa_\mu V$.
\end{proof}

\subsection{The control problem in  weak formulation}
Since the value function depends on the path of the state process $X$ anyway, we shall work on path-dependent setting directly, i.e. we will allow $b$, $\sigma$, $f$, and $g$ to depend on the paths of $X$, namely $b, \si, f$ are functions on $[0,T]\times \Omega \times A$ and $g$ is a function on $\Omega$, so as to have a more general result.
Let $\cA^\ch_t$ denote the set of $\dbF$-progressively measurable $A$-valued processes $\a$ on $[t, T]$ such that $\a_s$ is $\cF_{(s-\ch)^+}$-measurable, namely $\a_s = \a_s(X_{[0, (s-\ch^+]})$. Given $(t,\mu)\in  \Th_T$ and $\a\in \cA^\ch_t$, denote by $\dbP^{t,\mu, \a}$ the unique probability measure  $\dbP\in \cP_2(\cF_T)$ such that $\dbP_{[0,t]} =\mu_{[0,t]} $ and $\dbP$ is the  strong solution of the following SDE on $[t, T]$:
\bea
\label{X2}
d X_s = b(s, X_\cd, \a_s) ds  + \si(s, X_\cd, \a_s)dW_s,~ t\le s\le T, \dbP\mbox{-a.s.}
\eea
We emphasize again that at above $\a_s = \a_s(X_{[0,(s-\ch^+]})$. We then define
\bea
\label{Vtmu2}
V(t,\mu) := \sup_{\a\in \cA^\ch_t} J(t, \mu,\a) :=  \sup_{\a\in \cA^\ch_t}  \dbE^{\dbP^{t,\mu, \a}}\Big[g(X_\cd) + \int_t^T f(s, X_\cd, \a_s) ds\Big].
\eea

\begin{remark}
\label{rem-cF0}
{\rm When $T\le \ch$, $\a_t$ is $\cF_0$-measurable for $t\in [0, T]$. Since $\cF_0$ is not degenerate here, so in general $\a$ may not be deterministic, and thus rigorously  speaking the formulation here is slightly different from that in Sections \ref{sect-deterministic} and \ref{sect-example}. However, they are equivalent when $\mu_0$ is degenerate, namely $X_0$ is a constant, $\mu$-a.s. 

Alternatively, following the rationale of information delay, one may require $\a_t$ to be $\cF_{0-} := \{\emptyset, \Omega\}$-measurable for $t<\ch$, and thus is deterministic. One minor disadvantage of this reformulation is that the information flow will have a jump at $t=\ch$. Again, this discontinuity disappears when $\mu_0$ is degenerate.
\qed}
\end{remark}

Similar to Assumption \ref{assum-deterministic1}, we shall assume:

\begin{assumption}
\label{assum-deterministic2} (i) $b, \si, f$  are $\dbF$-adapted, and %progressively measurable on $[0, T]\times \Om\times A$, and $g$ is Borel measurable on $\Om$. Moreover,  
$b(t, 0, a)$, $\si(t,0, a)$,  and $f(t,0,a)$ are bounded;

(ii) $b$ and $\si$ are uniformly Lipschitz continuous in $\o$,  uniformly continuous in $t$, and continuous in $a$;

(iii)  $f$ is uniformly continuous in $(t,\o)$ and continuous in $a$, and $g$ is uniformly continuous in $\o$. 
\end{assumption}

 Under the above assumptions, it is clear that \reff{X2} is wellposed, $V$ is $\dbF$-adapted, and analogous to Theorem \ref{thm-DPP1} one can easily prove
\bea
\label{pathDPP}
V(t,\mu) = \sup_{\a\in \cA^\ch_t}  \Big[V(t+\d, \dbP^{t,\mu,\a} )+ \int_t^{t+\d} \dbE^{\dbP^{t,\mu, \a}}\big[f(s, X_\cd, \a_s) \big]ds\Big].
\eea
Now assume $V\in C^{1,2}(\Th_T )$  in the sense of Definition \ref{defn-C12path}. By \reff{pathDPP}, similar to Theorem \ref{thm-value} one can easily derive 
\bea
\label{master-path}
 \pa_t V(t, \mu)  +  H(t,\mu, \pa_\mu V, \pa_\o\pa_\mu V) = 0,
 \eea
where, for $p: \Th_T\times \Om\to \dbR^d$ and $q:  \Th_T\times \Om\to \dbR^{d\times d}$,
\bea
\label{Hamiltonian2}
\left.\ba{c}
\dis H(t,\mu, p,q) := \sup_{\a\in \cA^\ch_t}  h(t, \mu, p,q, \a_t),\\ 
\dis h(t,\mu, p, q, \a_t):= \dbE^\mu\Big[ \big[b(\cd)\cd p(t, \mu, X_\cd) + {1\over 2}\si\si^\top(\cd): q(t, \mu, X_\cd) + f(\cd)\big](t, X_\cd, \a_t)\Big].
\ea\right.
\eea
Note that $\a_t$ is $\cF_{(t-\ch)^+}$-measurable. Denote $\ol t := (t-\ch)^+$ and let $\mu^{\ol t, \o}$ denote the regular conditional probability distribution of $\mu$ given $\cF_{\ol t}$, i.e. $\mu^{\ol t,\o}(E) = \dbE^\mu[1_E(X_{[0,t]})$ $\ | \ \cF_{\ol t}](\o)$ for $\mu$-a.e. $\o\in\Omega$. 
Then, 
\bea
\label{Hamiltonian2-1}
\left.\ba{c}
\dis h(t,\mu, p,q, \a_t):= \dbE^\mu\Big[ \ol h\big(t, X_{[0, \ol t]},  \mu, p, q, \a_t(X_{[0,\ol t]}\big)\big) \Big],\q\mbox{where}\ms\\ 
\dis \ol h(t, \o, \mu, p, q, a) := \dbE^{\mu^{\ol t, \o}}\Big[ \big[b(\cd)\cd p(t, \mu, X_\cd)  + {1\over 2}\si\si^\top(\cd): q(t, \mu, X_\cd) + f(\cd)\big](t, X_\cd, a)\Big].
\ea\right.
\eea
We remark that in \reff{Hamiltonian2} $h$ depends on the whole random variable $\a_t$, while in \reff{Hamiltonian2-1} $\ol h$ depends on the realized value $a\in A$.
We have the following result:

\begin{theorem}
\label{thm-Hamiltonian2-2}
Let Assumption \ref{assum-deterministic2} hold. 

(i) For any $p: \Th_T \times \dbR^d\to \dbR^d$, $q: \Th_T \times \dbR^d\to \dbR^{d\times d}$ uniformly continuous,  the Hamiltonian $H$ in \reff{Hamiltonian2} becomes
\bea
\label{Hamiltonian2-2}
 H(t,\mu, p,q) = \dbE^\mu\Big[ \sup_{a\in A}  \ol h(t,  X_{[0,\ol t]}, \mu, p, q, a)\Big];
\eea

(ii) Assume $V\in C^{1,2}(\Th_T )$. Then $V$ is the value function in \reff{Vtmu2} if and only if $V$ satisfies the following path-dependent master equation:
\bea
\label{masterpath-2}
\qq \pa_t V(t, \mu)  +  \dbE^\mu\Big[ \sup_{a\in A}  \ol h(t,  X_{[0,\ol t]}, \mu, \pa_\mu V, \pa_\o\pa_\mu V, a)\Big] = 0,~V(T,\mu) = \dbE^\mu[g(X_\cd)].
 \eea
\end{theorem}
\begin{proof}  $(i)$ Define
\beaa
\widetilde H(t,\mu, p,q) := \dbE^\mu\Big[ \sup_{a\in A}  \ol h(t,  X_{[0,\ol t]}, \mu, p, q, a)\Big].
\eeaa
It is clear that $H \le \widetilde H$.  To see the opposite inequality, fix $(t,\mu, p, q)$ as specified in $(i)$.  By our conditions, it is obvious that $\o \mapsto  \ol h(t,  \o, \mu, p, q, a)$ is $\cF_{\ol t}$-measurable for each $a$, and $a   \mapsto  \ol h(t,  \o, \mu, p, q, a)$ is continuous for each $\o$. Then $(\o, a)\mapsto  \ol h(t,  \o, \mu, p, q, a)$ is $\cF_{\ol t} \times \cB(A)$-measurable. By the standard measurable selection theorem, see e.g. \cite[Proposition 2.21]{EKT}, for any $\e>0$, there exists  an $\cF^\mu_{\ol t}$-measurable random variable $a^\e$ such that 
\beaa
 \ol h(t,  X_{[0,\ol t]}, \mu, p, q, a^\e) \ge \sup_{a\in A}  \ol h(t,  X_{[0,\ol t]}, \mu, p, q, a) - \e,\q\mu\mbox{-a.s,}
 \eeaa
where $\cF^\mu_{\ol t}$ denotes the $\mu$-augmentation of $\cF_{\ol t}$. By \cite[Proposition 1.2.2]{Zhang_book}, there exists $\cF_{\ol t}$-measurable $\a^\e_t$ such that $\a^\e_t = a^\e$, $\mu$-a.s. Then
\beaa
\widetilde H(t,\mu, p,q) \le  \dbE^\mu\big[ \ol h(t,  X_{[0,\ol t]}, \mu, p, q, \a^\e_t)\big] + \e \le H(t,\mu, p,q)+\e.
\eeaa
By the arbitrariness of $\e$, we obtain $\widetilde H \le H$, and thus the equality holds.

$(ii)$ follows from similar arguments as in Theorem  \ref{thm-value}. 
\end{proof} 
 
Assume further that the following Hamiltonian $\ol H$ has an optimal argument $a^*$:
 \bea
 \label{olH}
 \ol H(t, \o, \mu,\pa_\mu V, \pa_\o\pa_\mu V)   := \sup_{a\in A} \ol h(t, \o, \mu, \pa_\mu V, \pa_\o\pa_\mu V, a).
 \eea
 By \reff{Hamiltonian2-1}, we see that $a^*$ takes the form $I(t,  \mu^{\ol t, \o}, \o_{[0,\ol t]})$. Then \reff{X2} becomes a McKean-Vlasov SDE again:
 \bea
\label{X2-2}
\qq\q d X^*_s = b\big(s, X^*_\cd,  I(s, \dbP^{\ol s, X^*}, X^*_{[0,\ol s]})\big) ds + \si\big(s, X^*_\cd,  I(s, \dbP^{\ol s, X^*}, X^*_{[0,\ol s]})\big) dW_s,  \dbP\mbox{-a.s.} 
\eea
Similar to Theorem \ref{thm-control}, one can easily prove: 
\begin{theorem}
\label{thm-control2}
Let Assumption \ref{assum-deterministic2} hold and $V\in C^{1,2}(\Th_T )$ be the classical solution to the master equation \reff{masterpath-2}. Assume further that

\begin{enumerate}[(i)]

\item the Hamiltonian $\ol H$ defined by \reff{olH} has an optimal control $a^* = I(t,  \mu^{\ol t, \o}, \o_{[0,\ol t]})$, for any $(t,\mu)\in \Th_T $, where $I: \Th_T \times\Om  \to A$ is measurable;

\item for a fixed $(t,\mu)\in \Th_T $, the McKean-Vlasov SDE \reff{X2-2} on $[t, T]$ has a (strong) solution $\dbP^*$ such that $\dbP^*_{[0,t]}  = \mu_{[0,t]}$.

%\item $\a^* \in \cA_t$, where $\a^*_s := I(s, \cL_{X^*_s})$, $s\in [t, T]$.  

\end{enumerate}

\no Then $\a^*_s := I(s, (\dbP^*)^{\ol s, \o}, \o_{[0,\ol s]})$, $s\in [t, T]$,  is an optimal control for the optimization problem \reff{Vtmu2} with this fixed $(t,\mu)$. 
\end{theorem}

It will be interesting to extend Theorems \ref{thm-control-sufficient} and \ref{thm-classical} to this case. This requires the measurability and/or regularity in terms of the paths and is more challenging. We shall leave a more systematic study on these issues in future research. In the subsection below, we shall solve the linear quadratic case which extends the example in Subsection  \ref{sect-eg}.

Finally, consider a special case where $b, \si, f$ do not depend on $X$. Then
\beaa
 \ol h(t, \o, \mu, p, q, a) &=&  {1\over 2}\si\si^\top(t, a): \dbE^\mu[q(t, \mu, X_\cd)|\cF_{\ol t}]  +b(t, a)\cd  \dbE^\mu[p(t, \mu, X_\cd)|\cF_{\ol t}] + f(t, a),
\eeaa
thus $a^*$ takes the form: $a^* = I\big(t, \dbE^\mu[p(t, \mu, X_\cd)|\cF_{\ol t}] , \dbE^\mu[q(t, \mu, X_\cd)|\cF_{\ol t}]\big)$. Therefore, \reff{X2-2} becomes:
\bea
\label{X2-3}
\left.\ba{lll}
\dis d X^*_s = b\big(s, X^*_\cd,  I\big(s, \dbE[\pa_\mu V|\cF_{\ol s}] , \dbE[\pa_\o\pa_\mu V|\cF_{\ol s}] \big)\big) ds\ms\\
\dis ~ + \si\big(s, X^*_\cd,  I\big(s, \dbE[\pa_\mu V|\cF_{\ol s}] , \dbE[\pa_\o\pa_\mu V|\cF_{\ol s}] \big)\big)dW_s, \q \dbP\mbox{-a.s,}
\ea\right.
\eea
where $\pa_\mu V$ and $\pa_\o \pa_\mu V$ are computed at $(s, \cL_{X^*_{[0,s]}}, X_s)$.

\subsection{The linear-quadratic example}
%The wellposedness of the general nonlinear path-dependent master equation \reff{master-path}-\reff{Hamiltonian2-2}-\reff{Hamiltonian2-1} is of course very challenging and we shall leave it for future research.  
Consider the path-dependent setting of  the example in Section \ref{sect-example} :
\bea
\label{example2}
\q \q d=1,~ A =\dbR,~ b(t, x, a) = a,~ \si = 1, ~ f(t,x, a) = -{1\over 2} a^2,~ g(x)=x^2,~ T= 2\ch.
\eea
In this case \reff{master-path} becomes: recalling $\ol t :=( t-\ch)^+$,
\bea
\label{master-path3}
\q\left.\ba{lll}
\dis \pa_t V(t, \mu)  +  {1\over 2} \dbE^\mu\Big[\pa_\o \pa_\mu V(t, \mu, X_\cd)\Big]  + {1\over 2} \dbE^\mu \Big[\Big|\dbE^\mu\big[\pa_\mu V(t, \mu, X_\cd)|\cF_{\ol t}\big]\Big|^2\Big] =0,\ms\\ 
\dis V(T,\mu) = \dbE^\mu[|X_T|^2].
 \ea\right.
 \eea
Moreover, provided \reff{master-path3} has a classical solution, then \reff{X2-3} reduces to:
\bea
\label{X2-4}
d X^*_s =  \dbE\big[\pa_\mu V(s,  \cL_{X^*_{[0,s]}}, X^*_s)|\cF_{\ol s}\big] ds+ dW_s, \q \dbP\mbox{-a.s.}
\eea

We shall show that \reff{master-path3} has a classical solution $V$ and that it is indeed path-dependent.

\begin{theorem}
\label{thm-eg2}
Let \reff{example2} hold. Assume  $\ch <{1\over 4}$ and denote $\ol \ch:= {1\over 2} - \ch$. 

(i)  The $V$ defined by \reff{Vtmu2} is equal to
\bea
\label{lq-Vu}
 V(t,\mu) = \left\{\ba{lll} 
\dis  \dbE^\mu\Big[|X_t|^2 +  \int_{t-\ch}^\ch { |\dbE^\mu_s[X_t]|^2\over 2  (\ol \ch +s)^2}   ds\Big] + T-t,\q t \in [\ch, 2\ch];\ms\ms \\
\dis  {\dbE^{\mu} [|X_t|^2 ] \over 2(\ol \ch +t)}  + \int_0^t {  \dbE^\mu[ |\dbE^\mu_s[X_t] |^2]\over 2  (\ol \ch +s)^2}   ds + \ch  +{1\over 2} \ln {  1\over 2  (\ol \ch+t)}\ms\ms\\
\dis\qq\qq +{\ch -t\over 2\ol \ch ({1\over 2}-2\ch +t)} \dbE^\mu\Big[\big|\dbE^\mu_0[X_t]\big|^2\Big], \q t \in [0, \ch).
 \ea\right.
\eea

\no It is in $C^{1,2}(\Th_T )$ and is  a classical solution to the path-dependent master equation \reff{master-path3};

(ii) For any $(t,\mu)\in \Th_T $,  the SDE \reff{X2-4} on $[t, T]$ with initial condition $\dbP\circ (X^*_{[0,t]})^{-1} = \mu_{[0,t]}$ has a strong solution $X^*$, and the optimal control takes the form:
\bea
\label{optimal3}
\a^*_s = \dbE^{\dbP^*}\big[\pa_\mu V(s, \dbP^*, X_\cd)|\cF_{(s-\ch)^+}\big],\q\mbox{where}\q \dbP^* = \dbP\circ (X^*)^{-1}.
\eea
\end{theorem}
\begin{proof}  We proceed in several steps.  Recall that $\ol t := (t-\ch)^+$.

{\it Step 1.} We first prove \reff{lq-Vu} for $t\in [\ch, 2\ch]$. Let $\mu\in \cP_2(\cF_T)$, $\a\in \cA^\ch_t$. Denote $\b_s := \a_{s+\ch}$, $s\in [0, \ch]$. Then $\b$ is $\dbF$-progressively measurable. Denote by $\mu\otimes_t \dbP_0$ the probability measure $\dbP$ such that $\dbP_{[0,t]}=\mu_{[0, t]}$ and $X_s =X_t+ W_s-W_t$, $\dbP$-a.s. for a $\dbP$-Borwnian motion $W$. By \reff{X2} and \reff{Vtmu2} one can easily see that
\bea
\label{Jua}
J(t,\mu,\a) &=& \dbE^{\mu\otimes_t \dbP_0}\Big[\Big|X_t + \int_{\ol t}^\ch \b_r dr + W_T-W_t\Big|^2 - {1\over 2} \int_{\ol t}^\ch \b_r^2 dr\Big]  \nonumber\\
&=& \dbE^\mu\Big[\Big|X_t + \int_{\ol t}^\ch \b_r dr \Big|^2 - {1\over 2} \int_{\ol t}^\ch \b_r^2 dr\Big] + T-t.
\eea
Since $\ch<{1\over 4}$ and $\b$ is $\dbF$-progressively measurable, one can easily show that the optimal $\b^*$ is the unique fixed point satisfying:
\beaa
\b^*_s = 2 \dbE^\mu_s\Big[X_t + \int_{\ol t}^\ch \b^*_r dr\Big],\q \ol t \le s\le \ch.
\eeaa
This implies that $\b^*$ is a $\mu$-martingale. Then
\beaa
\b^*_s =  2\dbE^\mu_s[X_t] + 2\int_{\ol t}^s \b^*_r dr +  2[\ch-s]\b^*_s.
\eeaa
 Solving this ODE, we obtain:
\bea
\label{beta*}
\b^*_s = \int_{\ol t}^s {\dbE^\mu_r[X_t]\over (\ol \ch+r)^2}dr + {\dbE^\mu_s[X_t]\over \ol \ch+s}.
\eea
Then by \reff{Vtmu2} and \reff{Jua} we have 
\bea
\label{Vtmu3}
V(t,\mu) = \dbE^\mu\Big[\big|X_t + \int_{\ol t}^\ch \b^*_sds \big|^2  - {1\over 2} \int_{\ol t}^\ch |\b^*_s|^2 ds\Big] + T-t.
\eea
Note that \reff{beta*} implies
\bea
\label{beta*int}
\int_{\ol t}^s \b^*_r dr = \left(\ol \ch+s\right)\int_{\ol t}^s {\dbE^\mu_r[X_t]\over  (\ol \ch+r)^2} dr,
\eea
and thus
\beaa
&& \dbE^\mu\Big[\Big|X_t + \int_{\ol t}^\ch \b^*_s ds \Big|^2 \Big] = \dbE^\mu\Big[\Big|X_t + \int_{\ol t}^\ch {\dbE^\mu_s[X_t]\over 2 (\ol \ch+s)^2} ds \Big|^2 \Big] \\
 &&=  \dbE^\mu\Big[|X_t|^2 + X_t \int_{\ol t}^\ch {\dbE^\mu_s[X_t]\over  (\ol \ch+s)^2} ds + \big|\int_{\ol t}^\ch {\dbE^\mu_s[X_t]\over  2(\ol \ch+s)^2} ds \big|^2 \Big] \\
 &&= \dbE^\mu\Big[|X_t|^2 + \int_{\ol t}^\ch {|\dbE^\mu_s[X_t]|^2\over  (\ol \ch+s)^2} ds +\int_{\ol t}^\ch \int_s^\ch {|\dbE^\mu_s[X_t]|^2\over  2(\ol \ch+s)^2(\ol \ch +r)^2}   dr ds  \Big]\\
 &&= \dbE^\mu\Big[|X_t|^2 + \int_{\ol t}^\ch {|\dbE^\mu_s[X_t]|^2\over  2(\ol \ch+s)^3} ds\Big].
 \eeaa
 Moreover, note that
 \beaa
 &&\dbE^\mu[|\b^*_s|^2]  = \dbE^\mu\Big[{|\dbE^\mu_s[X_t]|^2\over (\ol \ch+s)^2} + {2\dbE^\mu_s[X_t]\over \ol \ch+s}\int_{\ol t}^s {\dbE^\mu_r[X_t]\over ({1\over 2}-H+r)^2 }dr + \big|\int_{\ol t}^s {\dbE^\mu_r[X_t]\over (\ol \ch +r)^2}\big|^2 dr \Big] \\
 &&= \dbE^\mu\Big[{|\dbE^\mu_s[X_t]|^2\over (\ol \ch +s)^2} + {2\over \ol \ch +s}\int_{\ol t}^s {|\dbE^\mu_r[X_t]|^2\over (\ol \ch+r)^2}dr + 2\int_{\ol t}^s {|\dbE^\mu_r[X_t]|^2\over (\ol \ch +r)^2}\big[{ 1\over \ol \ch +r} - { 1\over \ol \ch +s}\big]dr \Big] \\
 &&=\dbE^\mu\Big[{|\dbE^\mu_s[X_t]|^2\over (\ol \ch +s)^2}  + 2\int_{\ol t}^s {|\dbE^\mu_r[X_t]|^2\over (\ol \ch+r)^3}dr \Big], 
 \eeaa
 and thus
 \beaa
 \dbE^\mu\Big[\int_{\ol t}^\ch |\b^*_s|^2ds\Big]  &=&  \dbE^\mu\Big[\int_{\ol t}^\ch {|\dbE^\mu_s[X_t]|^2\over (\ol \ch +s)^2}  ds +2 \int_{\ol t}^\ch {(\ch -s)|\dbE^\mu_s[X_t]|^2\over (\ol \ch+s)^3}  ds  \Big]  \\
 &=&\dbE^\mu\Big[ \int_{\ol t}^\ch {({1\over 2}+\ch -s)|\dbE^\mu_s[X_t]|^2\over (\ol \ch +s)^3}  ds  \Big].
 \eeaa
Plug these into \reff{Vtmu3}, we obtain 
\beaa
V(t,\mu) = \dbE^\mu\Big[|X_t|^2 + \int_{\ol t}^\ch {|\dbE^\mu_s[X_t]|^2\over 2 (\ol \ch+s)^3} ds - \int_{\ol t}^\ch {({1\over 2}+\ch-s)|\dbE^\mu_s[X_t]|^2\over 2(\ol \ch+s)^3}  ds\Big] + T-t,
\eeaa
which implies the first equality of  \reff{lq-Vu}  immediately.

{\it Step 2.} We next prove  \reff{lq-Vu}  for $t<\ch$. Set $t=\ch$ in Step 1, we have
\beaa
V(\ch,\mu) =   \dbE^\mu\Big[|X_\ch|^2 +  \int_0^\ch { |\dbE^\mu_s[X_\ch]|^2\over 2  (\ol \ch+s)^2}   ds\Big] + \ch.
\eeaa
Fix  $\mu \in \cP_2(\cF_T)$ and recall  $\mu\otimes_t \dbP_0$.  Note that  $\a\in \cA^\ch_t$ is $\cF_0$-measurable. Then, 
\beaa
V(\ch, \dbP^{t,\mu, \a}) &=& \dbE^{\mu\otimes_t \dbP_0} \Big[|X_t + \int_t^\ch \a_r dr + W_\ch -W_t|^2 \\
&&\qq +  \int_0^\ch { |\dbE^{\mu\otimes_t \dbP_0}_s[X_t + \int_t^\ch \a_r dr + W_\ch -W_t]|^2\over 2  (\ol \ch +s)^2}   ds\Big] + \ch\\
&=&\dbE^{\mu} \Big[|X_t + \int_t^\ch \a_r dr|^2\Big] + 2\ch -t \\
&&+  \dbE^\mu\Big[\int_0^t { |\dbE^\mu_s[X_t] + \int_t^\ch \a_r dr|^2\over 2  (\ol \ch+s)^2}   ds + \int_t^\ch { |X_t + \int_t^\ch \a_r dr|^2 + s-t\over 2  (\ol \ch+s)^2}   ds  \Big] \\
&=& \G_t +{1\over 2\ol \ch }\dbE^\mu\Big[2\dbE^\mu_0[X_t]\int_t^\ch \a_r dr + \big|\int_t^\ch  \a_rdr\big|^2\Big],\\
\mbox{where}&& \G_t := {\dbE^{\mu} [|X_t|^2 ] \over 2(\ol \ch+t)}  + \int_0^t {  \dbE^\mu[ |\dbE^\mu_s[X_t] |^2]\over 2  (\ol \ch+s)^2}   ds + \ch  + {1\over 2}\ln {  1\over 2  (\ol \ch +t)} . 
\eeaa
By the DPP \reff{pathDPP}, we have
\beaa
V(t,\mu) &=& \sup_{\a\in \cA^\ch_t} \Big[V(\ch, \dbP^{t,\mu, \a}) - {1\over 2} \int_t^\ch \dbE^\mu[\a_s^2] ds\Big] \\
&=&\G_t+ \sup_{\a\in \cA^\ch_t}{1\over 2\ol \ch} \dbE^\mu \Big[2\dbE^\mu_0[X_t]\int_t^\ch \a_r dr + \big|\int_t^\ch \a_rdr\big|^2 - \ol\ch \int_t^\ch \a_s^2 ds\Big].
\eeaa
One can easily see that the optimal $\a^*$ satisfies:
\beaa
\dbE^\mu_0[X_t] +  \int_t^\ch  \a^*_r dr =\ol\ch \a_s^*.
\eeaa
This implies 
\bea
\label{alpha*2}
\a^*_s = c_t:= {\dbE^\mu_0[X_t]\over \ol \ch +\ol t}, \q t\le s \le \ch ,
\eea
and thus
\beaa
V(t,\mu) &=&\G_t+ {1\over 2\ol \ch} \dbE^\mu\Big[2\dbE^\mu_0[X_t]\int_t^\ch\a^*_r dr + \big|\int_t^\ch  \a^*_rdr\big|^2- \ol\ch \int_t^\ch |\a^*_s|^2 ds\Big],
\eeaa
which implies the second equality of \reff{lq-Vu} immediately. 

{\it Step 3.} We now verify that $V\in C^{1,2}(\Th_T)$ and satisfies \reff{master-path3}. First consider $t\in [\ch, 2\ch]$. 
By \reff{lq-Vu} one may verify straightforwardly that
\bea
\label{patV3}
 \pa_t V(t,\mu) &=& - {\dbE^\mu\big[ |\dbE^\mu_{\ol t}[X_t]|^2\big]\over 2  ( \ol \ch +\ol t)^2}   -1.
\eea
To see $\pa_\mu V$, we remark that the $V$ in \reff{lq-Vu} is very smooth and actually one can use the stronger definition in the sprit of \reff{pamuV2} instead of Definition \ref{defn-C12path}. We again refer to \cite{WZ} for details and will derive $\pa_\mu V$ formally.  Given a random variable $X'_t \in \dbL^2(\cF_t)$, in the sprit of \reff{pamuV2} we have
 \beaa
\dbE\left[ \pa_\mu V(t,\mu, X_\cd) \ X_t' \right] &=&  \dbE^\mu\Big[2X_tX'_t +  \int_{\ol t}^\ch { \dbE^\mu_s[X_t] \dbE^\mu_s[X'_t] \over   (\ol \ch+s)^2} ds\Big]\\
&=& \dbE^\mu\Big[\big[2X_t +  \int_{\ol t}^\ch { \dbE^\mu_s[X_t]  \over   (\ol \ch+s)^2} ds\big] X'_t\Big].
\eeaa
Then
\bea
\label{pamuV3}
 \pa_\mu V(t,\mu, X_\cd) = 2X_t+  \int_{\ol t}^\ch { \dbE^\mu_s[X_t] \over   (\ol \ch +s)^2}  ds.
 \eea
 Moreover, note that  $\int_{\ol t}^\ch { \dbE^\mu_s[X_t] \over   (\ol \ch+s)^2}  ds$ is $\cF_\ch$-measurable and $\ch\le t$, then its path derivative with respect to $\o$ is $0$, and thus 
 \bea
 \label{paomuV3}
 \pa_\o  \pa_\mu V(t,\mu, X_\cd) = 2.
\eea
Note that, by \reff{pamuV3},
\beaa
\dbE^\mu\big[\pa_\mu V(t, \mu, X_\cd)|\cF_{\ol t}\big] = \dbE^\mu_{\ol t}[ X_t]  \Big[ 2+ \int_{\ol t}^\ch {1 \over   (\ol \ch+s)^2}  ds\Big] = { \dbE^\mu_{\ol t}[ X_t] \over  \ol \ch +\ol t}.
\eeaa
Then
\beaa
&& \pa_t V(t, \mu)  +  {1\over 2} \dbE^\mu\Big[\pa_\o \pa_\mu V(t, \mu, X_\cd)\Big] + {1\over 2} \dbE^\mu \Big[\Big|\dbE^\mu_{\ol t}\big[\pa_\mu V(t, \mu, X_\cd)\big]\Big|^2\Big] \\
&=& \Big[ - {\dbE^\mu\big[ |\dbE^\mu_{\ol t}[X_t]|^2\big]\over 2  ( \ol \ch +\ol t)^2}   -1 \Big]+ {1\over 2} \ 2 +  {1\over 2}\dbE^\mu\Big[\Big| { \dbE^\mu_{\ol t}[ X_t] \over \ol \ch +\ol t}\Big|^2\Big] =0.
\eeaa
That is, $V$ satisfies \reff{master-path3} for $t\in [\ch, 2\ch]$.

Next consider $t\in [0, \ch)$. By \reff{lq-Vu} one may verify directly that
\beaa
\pa_t V(t,\mu) &=&- {\dbE^{\mu} [|X_t|^2 ] \over 2(\ol \ch +t)^2}  +  {  \dbE^\mu[ |X_t |^2]\over 2  (\ol \ch +t)^2}   -{  1\over 2  (\ol \ch +t)}-{\dbE^\mu\big[\big|\dbE^\mu_0[X_t]\big|^2\big]\over 2( \ol \ch +\ol t)^2} \\
&=&-{\dbE^\mu\big[\big|\dbE^\mu_0[X_t]\big|^2\big] \over 2( \ol \ch +\ol t)^2} - {  1\over 2(\ol \ch +t)};\\
\pa_\mu V(t, \mu, X_\cd) &=&{X_t  \over \ol \ch +t}  + \int_0^t { \dbE^\mu_s[X_t]\over   (\ol \ch +s)^2}   ds +{(\ch -t)\dbE^\mu_0[X_t]\over \ol \ch( \ol \ch +\ol t)} ;\\
\pa_\o \pa_\mu V(t, \mu, X_\cd) &=&{1 \over \ol \ch +t}.
\eeaa
Note that
\beaa
\dbE^\mu_0\Big[\pa_\mu V(t, \mu, X_\cd)\Big]  = {\dbE^\mu_0[X_t]  \over \ol \ch +t}  + \int_0^t { \dbE^\mu_0[X_t]\over   (\ol \ch +s)^2}   ds +{(\ch -t)\dbE^\mu_0[X_t]\over \ol \ch( \ol \ch +\ol t)}=  {\dbE^\mu_0[X_t]  \over  \ol \ch +\ol t}.
\eeaa
Therefore,
\beaa
&& \pa_t V(t, \mu)  +  {1\over 2} \dbE^\mu\Big[\pa_\o \pa_\mu V(t, \mu, X_\cd)\Big] + {1\over 2} \dbE^\mu \Big[\Big|\dbE^\mu_0\big[\pa_\mu V(t, \mu, X_\cd)\big]\Big|^2\Big] \\
&&= -{\dbE^\mu\big[\big|\dbE^\mu_0[X_t]\big|^2\big] \over 2( \ol \ch +\ol t)^2} - {  1\over 2(\ol \ch +t)} + {1\over 2} {1 \over \ol \ch +t} + {1\over 2}\dbE^\mu\Big[ \Big|{\dbE^\mu_0[X_t]  \over  \ol \ch +\ol t}\Big|^2\Big] =0.
 \eeaa
 That is, $V$ is a classical solution to \reff{master-path3} on $[0, H]$ as well.
 
 {\it Step 4.} Finally, we prove (ii).  Note that  in this case  \reff{X2-4} can be rewritten as:
\bea
\label{X2-5}
dX^*_s = \a^*_s ds + dW_s.
\eea
   If $t\ge \ch$, the optimal control is $\a^*_s = \b^*_{\ol s}$ for the $\b^*$ defined by \reff{beta*}:
   \bea
   \label{a*3}
   \a^*_s = \int_{\ol t}^{\ol s} {\dbE^\mu_r[X_t]\over (\ol \ch+r)^2}dr + {\dbE^\mu_{\ol s}[X_t]\over \ol \ch+\ol s},\q  t\le s\le 2\ch.
   \eea
   Then,  by  \reff{beta*int},  it follows from \reff{X2-5} that   
 \bea
 \label{X*3}
X^*_s = X_t+ [ \ol \ch +\ol s] \int_{\ol t}^{\ol s} {\dbE^\mu_r[X_t]\over  (\ol \ch+r)^2} dr+ W_s-W_t,\q t\le s\le 2\ch.
\eea
 
 Now assume $t<\ch$. For $s\in [t, \ch]$, we have $\a^*_s = c_t$, where $c_t$ is defined by \reff{alpha*2}. For $s\in (\ch, 2\ch]$, the optimal $\a^*$ is obtained through the optimization problem $V(\ch, \dbP^{t,\mu, c_t})$. By  using \reff{beta*}, \reff{beta*int},  \reff{X2-5},   it follows from direct calculation that
  \bea
 \label{X*4}
\qq&&  X^*_s = \left\{\ba{lll}
 \dis X_t + c_t [s-t] + W_s-W_t,\q t\le s\le \ch;\\ \\
 \dis X^*_\ch + ( \ol \ch +\ol s)\int_0^{\ol s} {\dbE^\mu_r[X_t] + c_t[\ch -t]\over (\ol \ch +r)^2}dr + W_s-W^\ch_s,~ \ch <s\le t+\ch ;\\ \\
 \dis X^*_{\ch} +( \ol \ch +\ol s)\Big[\int_0^t {\dbE^\mu_r[X_t] + c_t[\ch -t]\over (\ol \ch +r)^2}dr  \\
 \dis\qq\qq + \int_t^{\ol s} {X^*_r+ c_t[\ch -r]\over (\ol \ch +r)^2}dr\Big] +W_s- W_\ch,\q t+\ch <s\le 2\ch;
 \ea\right.\\
 \label{a*4}
\qq&& \a^*_s =
 \left\{\ba{lll}
 \dis c_t,\q t\le s\le \ch;\ms\\ \\
\dis  {\dbE^\mu_{\ol s}[X_t] + c_t[\ch -t] \over  \ol \ch +\ol s} + \int_0^{\ol s} {\dbE^\mu_r[X_t] + c_t[\ch -t]\over (\ol \ch +r)^2}dr,~ \ch <s\le t+\ch;\ms\\ \\
\dis {X^*_{\ol s} + c_t [2\ch -s]\over    \ol \ch +\ol s} +  \int_0^t {\dbE^\mu_r[X_t] + c_t[\ch -t]\over (\ol \ch +r)^2}dr \\ 
\dis\qq\qq\qq +  \int_t^{\ol s} {X^*_r + c_t (\ch-r) \over   (\ol \ch +r)^2}  dr,\q  t+\ch <s\le 2\ch.
 \ea\right.
 \eea
   This completes the proof.
 \end{proof}
 
  \begin{remark}
 \label{rem-example}
 {\rm Consider the  special case with $t=0$ and $\mu = \d_{x_0}$.
 
 (i)  By the above results,  we have
   \bea
   \label{V5}
 V_0 = V(0, \d_{x_0}) 
&=&{x_0^2 \over 2\ol \ch}   + \ch +{1\over 2} \ln {  1\over 2  \ol \ch}+{\ch \over 2\ol \ch({1\over 2}-2\ch)} x_0^2\\
&=& {x_0^2\over 1-4\ch} + \ch -{1\over 2} \ln(1-2\ch). \nonumber
\eea
The optimal control and the optimal state process are: noting  that  $c_0 = \displaystyle {x_0\over {1\over 2}-2\ch}$,
  \bea
  \label{a*5}
   \a^*_s &=&
 \left\{\ba{lll}
 \dis c_0,\q 0\le s\le \ch;\ms\\ \\
\dis {X_{s-\ch} + c_0 [2\ch-s]\over   \ol \ch +\ol s}   +  \int_0^{\ol s} {X_r + c_0 (\ch-r) \over   ( \ol \ch +r)^2}  dr,~ \ch <s\le 2\ch.
 \ea\right.\\
  \label{X*5}
 X^*_s &=& \left\{\ba{lll}
 \dis x_0 + c_0 s + W_s,\q 0\le s\le \ch;\\ \\
 \dis X^*_{\ch} +(\ol\ch+\ol s) \int_0^{\ol s} {X^*_r+ c_0[\ch -r]\over (\ol \ch +r)^2}dr + W_s - W_{\ch},~ \ch<s\le 2\ch,
 \ea\right.
 \eea

 (ii) If the delay time is $2\ch$ (and $T$ is still $2\ch$), namely considering only deterministic controls $\a$, then
 \beaa
 J(0, x_0, \a) &=& \dbE^{\dbP_0} \Big[\big|x_0 + \int_0^T \a_s ds + W_T\big|^2 - {1\over 2} \int_0^T |\a_s|^2ds\Big] \\
 &=& \big|x_0 + \int_0^T \a_s ds\big|^2  + T - {1\over 2} \int_0^T |\a_s|^2ds.
 \eeaa
 One can easily show that the optimal control and optimal values are, using superscript $2\ch$ to denote the delay time and recalling $T=2\ch$,
 \bea
 \label{V02H}
 \a^{2\ch}_s = {2x_0\over 1-4\ch}, \q  V^{2\ch}_0 = {x_0^2\over 1-4\ch} + 2\ch.
 \eea
 
 (iii) If the delay time is $0$, then we have a standard HJB equation:
 \beaa
 \pa_t v(t,x) + {1\over 2} \pa_{xx} v + {1\over 2} |\pa_x v|^2 =0,\q v(T,x) = |x|^2.
 \eeaa
 One can easily see that the above PDE has a classical solution 
 \beaa
 v(t,x) = {x^2 \over 1-2T + 2t} - {1\over 2} \ln (1-2T+2t).
 \eeaa
 This implies that, using superscript $0$ to denote the delay time $0$ and recalling $T=2\ch$,
 \bea
 \label{V00}
\qq \a^0(s,x) =\pa_x v(t,x) =  {2x\over 1-4\ch +2t}, ~  V^{0}_0 = v(0,x_0)={x_0^2\over 1-4\ch} - {1\over 2} \ln (1-4\ch).
 \eea
 
 (iv) One can verify straightforwardly that, for $0<\ch <{1\over 4}$, 
 \beaa
 2\ch < \ch-{1\over 2} \ln(1-2\ch) <  - {1\over 2} \ln (1-4\ch), \q\mbox{which implies}\q V^{2\ch}_0 < V_0 < V^0_0.
 \eeaa
This indicates that the information delay indeed decreases the value function, consistent with our intuition. 
 \qed}
 \end{remark}

\appendix

\section{}
\label{sect-appendix}
\setcounter{equation}{0}
In this appendix, we show heuristically how the stochastic maximum principle leads to the same structure as in Section \ref{sect-deterministic}. We remark that this approach has also been used by  \cite{HT} recently for a mixture of deterministic and stochastic controls in a linear quadratic setting. To focus on the main idea and simplify the presentation, we consider the following simple case with deterministic controls $\a\in \cA_0$:
\bea
\label{AppendixV0}
\left.\ba{c}
\dis V_0 := \sup_{\a\in \cA_0} J(\a) := \sup_{\a\in \cA} \dbE\Big[g(X^\a_T) + \int_0^T f(t, \a_t) dt\Big], \\
\mbox{where}~ X^\a_t = x + \dis \int_0^t b(s, \a_s) ds + W_t,
\ea\right.
\eea
where $\cA_0$ is the set of all Borel measurable functions $\a: [0, T] \to A$.  %In particular, we will see how the law of $X$ comes to play naturally.

Since $\cA_0$ is convex, namely, for $\a, \a'\in \cA$, we have $\a + \e(\a'-\a) \in \cA$ for all $\e\in (0,1)$.  Fix  $\a, \a'\in \cA$ and denote $\D \a := \a'-\a$, $\a^\e := \a + \e\D \a$.  Assume $b$ and $f$ are continuous differentiable in $a$ and $g$ is continuously differentiable in $x$. Then
\beaa
\td X_t &:=&\lim_{\e\to 0} {X^{\a^\e}_t - X^{\a}_t\over \e} = \int_0^t  \pa_a b(s, \a_s) \D \a_s ds,\\
\td J &:=&\lim_{\e\to 0} {J(\a^\e)-J(\a)\over \e} =\dbE\Big[ \pa_xg(X^{\a}_T) \td X_T + \int_0^t  \pa_a f(s, \a_s) \D \a_s ds \Big]. 
\eeaa
Let $(\widetilde Y^\a, \widetilde Z^\a)$ be the solution to the following BSDE:
\beaa
\widetilde Y^\a_t = \pa_xg(X^{\a}_T)  - \int_t^T \widetilde Z^\a_s dW_s.
\eeaa
We emphasize that $(\widetilde Y^\a, \widetilde Z^\a)$ depend on $\a$, but not on $\D \a$. Then
\bea
\label{tdJ}
\td J &=& \dbE\Big[ \int_t^T\big[\widetilde Y^\a_s   \pa_a b(s, \a_s)  + \pa_a f(s, \a_s)\big]  \D \a_s ds \Big] \nonumber\\
& = & \int_t^T \Big[\dbE[\widetilde Y^\a_s]   \pa_a b(s, \a_s)  + \pa_a f(s, \a_s)\Big]  \D \a_s ds \Big], 
\eea
where the second equality relies on the fact that $\a$ and $\D \a$ are deterministic. Now assume $\a^*\in \cA_0$ is an optimal argument, then $\D J\le 0$ for all possible $\D \a$. Assume further that $\a^*$ is an inner point of $\cA$ in the sense that one may choose $\D \a$ in all directions. Then 
\bea
\label{SMP}
\dbE[\widetilde Y^{\a^*}_t]   \pa_a b(t, \a^*_t)  + \pa_a f(t, \a^*_t) =0.
\eea
Assume $b$ and $f$ are such that the above equation determines a function $\widetilde I(t,x)$ such that $\a^*_t =\widetilde I(t, \dbE[\widetilde Y^{\a^*}_t] )$. Then, denoting $X^*:= X^{\a^*},\widetilde Y^*:= \widetilde Y^{\a^*}, \widetilde Z^*:= \widetilde Z^{\a^*}$, we obtain the following coupled forward backward  SDE:
\bea
\label{FBSDE*}
 X^*_t = x + \int_0^t b(s, \widetilde I(s, \dbE[\widetilde Y^*_s])) ds +  W_t, \qq \widetilde Y^*_t =   \pa_xg(X^*_T)  - \int_t^T \widetilde Z^*_s dW_s.
\eea
We emphasize that the above FBSDE is of McKean-Vlasov type because the forward one includes $\dbE[\widetilde Y^*_s]$, which is determined by the law of $\widetilde Y^*_s$ rather than the value of  $Y^*_s$.  Assume the above FBSDE is well-posed and we have the decoupling field: $\widetilde Y^*_t= \widetilde V(t, \cL_{X^*_t},  X^*_t)$, which without surprise involves the law of $X^*$.  Denote $I(t,\mu) := \widetilde I\big(t, \dbE[\widetilde V(t, \mu, \xi)]\big)$, where as usual $\cL_\xi = \mu$. Then $\widetilde I(t, \dbE[\widetilde Y^*_t]) =I(t, \cL_{X^*_t})$, and thus
\bea
\label{AppendixX*}
X^*_t = x + \int_0^t b(s, I(s,  \cL_{X^*_s})) ds +  W_t,
\eea
which is consistent with \reff{MV}.

\begin{remark}
\label{rem-SMP2}
{\rm When the control $\a_t$ is $\cF_t$-measurable, the first equality of \reff{tdJ} still holds but the second fails. Due to the arbitrariness of $\D \a$, in this case the first order condition \reff{SMP} becomes:
\bea
\label{SMP2}
\widetilde Y^{\a^*}_t   \pa_a b(t, \a^*_t)  + \pa_a f(t, \a^*_t) =0.
\eea
This leads to $\a^*_t= \widetilde I(t, \widetilde Y^{\a^*}_t)$ which in turn leads to a standard FBSDE. These are very standard arguments in the literature. Again, here due to our constraint of deterministic control, the optimal control $\a^*_t$ depends on $\dbE[\widetilde Y^*_t]$ instead of $\widetilde Y^*_t$, and hence depends on the law of $X^*_t$. 
}
\end{remark}

\section{}
\label{sect-appendix_comparison}
In this appendix, we will show some mathematical details of the discussion outlined in Section \ref{sec:review}. Specifically, we will describe some aspects of the noisy observation case and how it compares to ours. The state variable $X$ is governed by the dynamics (\ref{Ja}). For simplicity of notation, we assume $d=1$ and $\sigma \equiv 1$ in what follows. Then, 
 \bea
 \label{Ja-app}
 \left.\ba{c}
\dis X_s^{t, \xi, \a} = \xi + \int_t^s b(r, X_r^{t, \xi, \a}, \a_r)dr + W_s - W_t, ~ s\in [t, T],\\
\dis J(t,p, \a) := \dbE^p\left[g(X_T^{t, \xi, \a}) + \int_t^T f(s, X_s^{t, \xi, \a}, \a_s)ds \right],
\ea\right.
\eea
for $\xi$ with probability density $p$. Differently from the previous control problem, the agent observes a non-linear noisy process given by
$$Y_s = \int_t^s h(r, X_r^{t, \xi, \a}) dr + \widetilde{W}_s,$$
where $\widetilde W$ is a Brownian motion independent of $W$. Thus, an admissible control $\a$ has to be progressively measurable with respect to the filtration generated by $Y$, $\{\cF^Y_s\}_{s \in [t,T]}$. We will denote this space by $\widetilde \cA_{[t,T]}$. Hence, the value function is given by
\bea
\label{Vt-app}
V(t, p) := \sup_{\a\in \widetilde \cA_{[t,T]}} J(t,p, \a).
\eea

We will follow closely the approach of \cite{partial_obs_benes_karatzas}. First we introduce some notations. Given two functions $\f, \psi: \dbR^d \to \dbR$, denote $\la \varphi, \psi\ra := \int_{\dbR^d} \varphi(z) \psi(z) dz$.  Given a function $F : L^2(\dbR) \longrightarrow \dbR$, its derivative with respect to $p$ is a function   $\pa_p F(p): L^2(\dbR) \longrightarrow \dbR$, defined  in the G\^{a}teaux sense:
$$
 \left. {d \over d\e} F(p + \e \varphi) \right|_{\e = 0} =\la \pa_p F(p),  \varphi\ra,
 $$
for appropriate test function $\f: \dbR^d\to \dbR$. The second order derivative $\pa_{pp} F(p)$ is defined similarly through $\la \pa_{pp} F, [\f, \psi]\ra$ and can be viewed as a bilinear mapping. Moreover, $\pa_\mu$ and $\pa_p$ are related through the equation (see e.g. \cite{master_eq_bensoussan}):  for measure $\mu$ with density $p$,
\bea
 \label{eq:mu_p_der}
\partial_\mu F(\mu, x) = \partial_x \partial_p F(p)(x).
\eea

\cite{partial_obs_benes_karatzas}  show that dynamics of a proper unnormalized density of the distribution of $X_s^\a$ given $\cF_s^Y$, denoted by $\rho_s$, is given by
$$d\rho_s^{t,p}(x) = \cL^{\ast \a_s}_s \rho_s^{t,p}(x) dt + h(s,x) \rho_s^{t,p}(x) dY_s,$$
with $\rho_t^{t,p} = p$, which is the unnormalized density of $X_t$, and
\beaa
\cL^{\ast a}_s  &=& \frac{1}{2} \pa_{xx} - b(s,x,a) \pa_x  - \pa_x b(s,x,a).
\eeaa
Moreover, one may write
\beaa 
J(t, p, \a) = \dbE\left[\frac{\la g, \rho_T^{t,p}\ra}{\la 1, \rho_T^{t,p}\ra} + \int_t^T\frac{\la f(s, \cdot, \a_s), \rho_s^{t,p}\ra}{\la 1, \rho_s^{t,p}\ra} ds \right]. %\qV(t, p) = \sup_{\a\in \widetilde \cA_{[t, T]}} J(t,p, \a).
\eeaa
Under certain condition, $V$ satisfies the following HJB equation (see \cite[Equations (2.14)-(2.15)]{partial_obs_benes_karatzas}) with terminal condition $V(T,p) = \la g,p\ra$:
\bea \label{eq:hjb_mortensen}
\pa_t V(t, p) &+& \frac{1}{2} \big\la \pa_{pp} V (t,p), [h(t,\cdot)p, \ h(t,\cdot)p] \big\ra\\
&+& \sup_{a \in A} \Big[ \big\la \pa_p V(t,p), \cL^{\ast a}_t p\big\ra + \big\la f(t, \cdot, a), p\big\ra\Big] = 0. \nonumber
\eea

We would like to point out that the deterministic control problem studied in Section \ref{sect-deterministic} is equivalent to the noisy observation control problem with $h \equiv 0$, i.e. the pure noise case. Under this situation, we will now show that the master equation (\ref{master}) is the HJB equation (\ref{eq:hjb_mortensen}), when restricted to those measures with density.
In fact, in this case, the HJB equation \reff{eq:hjb_mortensen}  becomes
\bea \label{eq:hjb_mortensen_exam}
\pa_t V(t, p)  + \sup_{a \in A}  \Big[ \big\la \pa_p V(t,p), \cL^{\ast a}_t p\big\ra + \big\la f(t, \cdot, a), p\big\ra\Big]  = 0.
\eea
By using integrating by parts formula, we have
\beaa
\la \pa_p V(t,p), \pa_{xx} p\ra &=&\la  \pa_{xx} \pa_p V(t,p), p\ra;\\
\la  \pa_p V(t,p),  b(t,\cd,a) \pa_x p\ra&=& -\big\la \pa_x \pa_p V(t,p)  b(t,\cd,a) + \pa_p V(t,p) \pa_x b(t,\cd, a), ~ p\big\ra.
\eeaa
Then, for measure $\mu$ with density and by using \reff{eq:mu_p_der},
\beaa
\la \pa_p V(t,p), \cL^{\ast a}_t p\ra &=& \Big\la \pa_p V(t,p),~ \frac{1}{2} \pa_{xx} p - b(t,\cd,a) \pa_x p - \pa_x b(t,\cd, a)p\Big\ra\\
&=&  \Big\la {1\over 2} \pa_{xx} \pa_p V(t,p) +  b(t,\cd, a) \pa_x \pa_p V(t,p),~p\Big\ra\\
&=& \Big\la {1\over 2} \pa_{x} \pa_\mu V(t,\mu) +  b(t,\cd, a)  \pa_\mu V(t,\mu),~p\Big\ra.
\eeaa
Plug this into  \reff{eq:hjb_mortensen} we obtain our master equation \reff{master} immediately.

\clearpage

%\bibliography{bib_all}
\bibliographystyle{abbrvnat}
%\bibliography{E:/Dropbox/BibTex/bib_all}

\end{document}

%% file: SZ_short_sicon_shared.tex
\usepackage{amsmath}
\usepackage{amssymb}
\usepackage{latexsym}
\usepackage{color}
\usepackage{comment}
\usepackage{enumerate}
\usepackage{natbib}
\usepackage{breakcites}
\usepackage[titletoc,title]{appendix}
\usepackage{url}
\usepackage[normalem]{ulem}

\usepackage{anysize}
\marginsize{2.5cm}{2.5cm}{1.5cm}{4cm}

\newenvironment{proof}{\noindent {\bf Proof.\/}}{$\qed$\vskip 0.1in}

\newcommand{\la}{\langle}
\newcommand{\ra}{\rangle}

\newcommand{\ol}{\overline}

\newcommand{\e}{\varepsilon}

\newcommand{\ba}{\begin{array}}
\newcommand{\ea}{\end{array}}
\newcommand{\be}{\begin{equation}}
\newcommand{\ee}{\end{equation}}
\newcommand{\bea}{\begin{eqnarray}}
\newcommand{\eea}{\end{eqnarray}}
\newcommand{\beaa}{\begin{eqnarray*}}
\newcommand{\eeaa}{\end{eqnarray*}}

\def\dbE{\mathbb{E}}
\def\dbF{\mathbb{F}}
\def\dbG{\mathbb{G}}

\def\dbL{\mathbb{L}}

\def\dbP{\mathbb{P}}
\def\dbR{\mathbb{R}}

\def\a{\alpha}
\def\b{\beta}

\def\d{\delta}
\def\e{\varepsilon}

\def\si{\sigma}

\def\f{\varphi}

\def\o{\omega}

\def\G{\Gamma}
\def\D{\Delta}
\def\Th{\Theta}

\def\Om{\Omega}
\def\cA{{\cal A}}
\def\cB{{\cal B}}

\def\cD{{\cal D}}

\def\cF{{\cal F}}
\def\cG{{\cal G}}

\def\cL{{\cal L}}

\def\cP{{\cal P}}

\def\cW{{\cal W}}

\def\ch{\textsc{h}}

\def\no{\noindent}

\def\ms{\medskip}

\def\q{\quad}
\def\qq{\qquad}

\def\pa{\partial}
\def\cd{\cdot}
\def\cds{\cdots}

\def\td{\nabla}

\def\tr{\hbox{\rm tr}}

\def\qed{ \hfill \vrule width.25cm height.25cm depth0cm\smallskip}

\newcommand{\basa}{\begin{assumption}}
\newcommand{\easa}{\end{assumption}}

\newcommand{\bas}{\begin{assum}}
\newcommand{\eas}{\end{assum}}

\def\pa{\partial}

 \def\cd{\cdot}
\def\cds{\cdots}

\def\tr{\hbox{\rm tr$\,$}}

\def\dis{\displaystyle}

\def\cad{c\`{a}dl\`{a}g}

\def\1{{\bf 1}}

\def\:{\!:\!}
\def\reff#1{{\rm(\ref{#1})}}
\newtheorem{theorem}{Theorem}[section]
\newtheorem{lemma}[theorem]{Lemma}

\newtheorem{remark}[theorem]{Remark}
\newtheorem{example}[theorem]{Example}
\newtheorem{definition}[theorem]{Definition}
\newtheorem{assumption}[theorem]{Assumption}

\numberwithin{equation}{section}
\numberwithin{theorem}{section}

\newcommand{\TheTitle}{Stochastic Control with Delayed Information and Related Nonlinear Master Equation}

\title{{\TheTitle}}

\author{
  Yuri F. Saporito\thanks{Funda\c{c}\~ao Getulio Vargas, Escola de Matem\'atica Aplicada, yuri.saporito@fgv.br} 
  \and
  Jianfeng Zhang\thanks{University of Southern California, Department of Mathematics, jianfenz@usc.edu. Research supported in part by NSF grant DMS 1413717.}
}